\documentclass[11pt]{amsart}

\usepackage{amsmath,amssymb,latexsym}
\usepackage{dsfont}
\usepackage[T1]{fontenc}
\usepackage{enumerate}
\usepackage{eqnarray}
\usepackage{mathtools}
\usepackage{url}
\usepackage{xcolor}
\usepackage{lipsum}

\usepackage[all]{xy}

\usepackage[latin9]{inputenc}

\usepackage{pb-diagram,pb-xy}
\newcommand{\Th}{\operatorname{Th}}

\usepackage[a4paper,top=3cm, bottom=3cm, inner=3.5cm,outer=2.7cm,headsep=1.5cm]{geometry}
\numberwithin{equation}{section}

\newcounter{casenum}

\newcommand\blfootnote[1]{%
  \begingroup
  \renewcommand\thefootnote{}\footnote{#1}%
  \addtocounter{footnote}{-1}%
  \endgroup
}

\theoremstyle{plain}
\newtheorem{theorem}{Theorem}[section]
\newtheorem{prop}[theorem]{Proposition}
\newtheorem{fact}[theorem]{Fact}
\newtheorem{lemma}[theorem]{Lemma}
\newtheorem{cor}[theorem]{Corollary}
\newtheorem{claim}[theorem]{Claim}

\theoremstyle{definition}
\newtheorem{defn}[theorem]{Definition}
\newtheorem{remark}[theorem]{Remark}

\newtheorem{problem}[theorem]{Problem}
\newtheorem{expl}[theorem]{Example}

\newcommand{\tp}{\operatorname{tp}}

\newcommand{\ac}{\operatorname{ac}}
\newcommand{\bdd}{\operatorname{bdd}}

\newcommand{\DLO}{\operatorname{DLO}}

\newcommand{\acl}{\operatorname{acl}}

\newcommand{\dcl}{\operatorname{dcl}}

\newcommand{\eq}{\operatorname{eq}}
\newcommand{\lex}{\operatorname{lex}}
\newcommand{\indu}{\operatorname{ind}}

\newcommand{\revlex}{\operatorname{revlex}}

\newcommand{\VC}{\operatorname{VC}}

\newcommand{\CP}{\operatorname{P}}

\newcommand{\lift}{\operatorname{lift}}
\newcommand{\Aut}{\operatorname{Aut}}
\newcommand{\ACF}{\operatorname{ACF}}
\newcommand{\ACVF}{\operatorname{ACVF}}

\newcommand{\CH}{\operatorname{CH}}

\title{Semi-equational theories}

\author{Artem Chernikov}
\address{Department of Mathematics, University of California, Los Angeles, Los Angeles, CA 90095, USA}
\email{chernikov@math.ucla.edu}

\author{Alex Mennen}
\address{Department of Mathematics, University of California, Los Angeles, Los Angeles, CA 90095, USA}
\email{alexmennen@math.ucla.edu}

\begin{document}

\def\Ind#1#2{#1\setbox0=\hbox{$#1x$}\kern\wd0\hbox to 0pt{\hss$#1\mid$\hss}
\lower.9\ht0\hbox to 0pt{\hss$#1\smile$\hss}\kern\wd0}
\def\Notind#1#2{#1\setbox0=\hbox{$#1x$}\kern\wd0\hbox to 0pt{\mathchardef
\nn="3236\hss$#1\nn$\kern1.4\wd0\hss}\hbox to 0pt{\hss$#1\mid$\hss}\lower.9\ht0
\hbox to 0pt{\hss$#1\smile$\hss}\kern\wd0}
\def\indi{\mathop{\mathpalette\Ind{}}}
\def\nindi{\mathop{\mathpalette\Notind{}}}

\global\long\def\ind{\operatorname{\indi}}

\global\long\def\nind{\operatorname{\nindi}}

\begin{abstract}
	We introduce and study (weakly) semi-equational theories, generalizing equationality in stable theories (in the sense of Srour) to the NIP context. In particular, we establish a connection to distality via one-sided strong honest definitions; demonstrate that certain trees are semi-equational, while algebraically closed valued fields are not weakly semi-equational; and obtain a general criterion for weak semi-equationality of an expansion of a distal structure by a new  predicate.
	\end{abstract}

\maketitle
\section{Introduction}
\blfootnote{This version of the article was significantly shortened for the journal publication, resulting in some details being omitted. For the full version of the article see \cite{semieqExp}.}
Equations and equational theories were introduced by Srour \cite{srour1988notion1, srour1988notion2, srour1990notion} in order to distinguish ``positive'' information in an arbitrary first order theory, i.e.~to find a well behaved class of ``closed'' sets among the definable sets, by analogy to the algebraic sets among the constructible ones in algebraically closed fields. We recall the definition:

\begin{defn}
	\begin{enumerate}
		\item A partitioned formula $\varphi(x,y)$, with $x,y$ tuples of variables, is an \emph{equation} (with respect to a first-order theory $T$) if there do not exist $\mathcal{M} \models T$ and tuples $\left( a_i, b_i : i \in \omega \right)$ in $\mathcal{M}$ such that 
		$\mathcal{M} \models\varphi\left(a_{i},b_{j}\right)$ for all $j<i$ and $\mathcal{M} \models \neg \varphi\left(a_{i},b_{i}\right)$ for all $i$.
		\item 	A theory $T$ is \emph{equational} if every formula $\varphi(x,y)$, with $x,y$ arbitrary finite tuples of variables, is equivalent in $T$ to a Boolean combination of finitely many equations $\varphi_1(x,y), \ldots, \varphi_n(x,y)$.
	\end{enumerate}
\end{defn}
\noindent It is immediate from the definition that every equational theory is stable. Structural properties of equational theories in relation to forking and stability theory are studied in \cite{pillay1984closed, hrushovski1989stable, junker2000note, junker2002theories, junker2001indiscernible}. Many natural stable theories are equational; \cite{hrushovski1989stable} provided the first example of a stable non-equational theory. More recently it was demonstrated that the stable theory of non-abelian free groups is not equational \cite{sela2012free, muller2017nonequational}, and further examples are constructed in \cite{martin2021trois}. It is demonstrated in \cite{martin2020equational} that all theories of  separably closed fields are equational (generalizing earlier work of Srour \cite{srour1986independence}). See also \cite{OHara} for an accessible introduction to equationality.

We propose a generalization of equations and equational theories to the larger class of NIP theories (see Section \ref{def: weak semi-eq} for a more detailed discussion):

\begin{defn}\label{def: weak semi-eq}
Let $T$ be a first-order theory and $\mathbb{M} \models T$ a monster model of $T$.
\begin{enumerate}
	\item A partitioned formula $\varphi(x,y)$ is a \emph{semi-equation} (in $T$)  if there is no  sequence $\left(a_{i},b_{i} : i \in \omega \right)$ with $a_i \in \mathbb{M}^{x}, b_i \in \mathbb{M}^{y}$  such
that for all $i,j \in \omega$, $\models\varphi\left(a_{i},b_{j}\right)\iff i\neq j$.
\item A (partitioned) formula	$\varphi\left(x,y\right)$ is a \emph{weak semi-equation} if there is no $b \in \mathbb{M}^y$ and an ($\emptyset$-)indiscernible sequence 
$\left(a_{i} : i \in \mathbb{Z} \right)$
with $a_i \in \mathbb{M}^{x}$ such that the subsequence $\left(a_{i} : i \neq 0 \right)$ is indiscernible
over $b$, $\models\varphi\left(a_{i},b\right)$ for all $i \neq 0$,
but $\models \neg \varphi\left(a_{0},b\right)$.
\item  A theory $T$ is (\emph{weakly}) \emph{semi-equational} if every formula $\varphi(x,y) \in \mathcal{L}$, with $x,y$ arbitrary finite tuples of variables, 
is a Boolean combination of finitely many (weak) semi-equations $\psi_1(x,y), \ldots, \psi_n(x,y) \in \mathcal{L}$.	
\end{enumerate}
\end{defn}
\noindent Semi-equations are in particular weak semi-equations, every weakly semi-equational theory is NIP, and in a stable theory all three notions coincide (see Proposition \ref{prop: semieq and stab}). Some parts of the basic theory of equations naturally generalize to (weak) semi-equations, 
but there are also some new phenomena and complications appearing outside of stability. In particular, weak semi-equationality provides a simultaneous generalization of equationality and distality, bringing out some curious parallels  between those two notions (see Section \ref{sec: strong honest defs}). In this paper we develop the basic theory of (weak) semi-equations, and investigate (weak) semi-equationality in some examples. We view this as a first step, and a large number of questions remain open and can be found throughout the paper. 

In Section \ref{def: weak semi-eq} we provide some  equivalent characterizations of (weak) semi-equationality in terms of indiscernibles. We discuss closure of (weak) semi-equations under Boolean combinations (Proposition \ref{prop: Bool combs}), reducts and expansions (Proposition \ref{prop: semieq reduct}). In Section \ref{sec: rel to NIP etc} we discuss  how (weak) semi-equationality relates to the more familiar notions: all weakly semi-equational theories are NIP, distal theories are weakly semi-equational, and in a stable theory a formula is an equation if and only if it is a (weak) semi-equation (Proposition \ref{prop: semieq and stab}). In Section \ref{sec: quant semieq 1-based} we introduce some quantitive parameters associated to semi-equations. This parameter is related to breadth (Definition \ref{def: breadth}) of the family defined by the instances of a formula, and we observe that a formula is a semi-equation if and only if the family of its instances has finite breadth (Proposition \ref{prop: semieq iff fin breadth}).
The case when this parameter is minimal, i.e.~$1$-semi-equations, provide a generalization of weakly normal formulas characterizing $1$-based stable theories (Proposition \ref{prop: weak norm iff 1-semieq stab}). Hence $1$-semi-equationality can be viewed as a form of ``linearity'', or ``$1$-basedeness'' for NIP theories. We discuss its connections to a different form of ``linearity'' considered in \cite{basit2021zarankiewicz}, namely basic relations and almost linear Zarankiewicz bounds (see Proposition \ref{prop: 1-semieq iff basic} and Remark \ref{rem: Zarank etc for 1-semieq}), observing  that $(2,1)$-semi-equational theories do not define infinite fields.

In Section \ref{sec: semieq examples} we consider some examples of semi-equational theories. In Section \ref{sec: ex omin} we show that an $o$-minimal expansion of a group is linear if and only if it is $(2,1)$-semi-equational. It remains open if the field of reals is semi-equational (Problem \ref{prob: reals semieq}). We demonstrate that arbitrary unary expansions of linear orders (Section \ref{sec: lin ords}) and many ordered abelian groups (Section \ref{sec: OAGS}) are $1$-semi-equational.  
In Section \ref{sec: trees} we demonstrate that the theory of infinitely-branching dense trees is semi-equational (Theorem \ref{thm: trees are semieq}), but not $1$-semi-equational (even after naming parameters, see Theorem \ref{thm: trees not 1-semieq} and Corollary \ref{cor: tree const not 1-semieq}). Semi-equationality of arbitrary trees remains open (Problem \ref{prob: trees semieq}). In Section \ref{sec: circ orders} we observe that dense circular orders are not semi-equational, but become $1$-semi-equational after naming a single constant (in contrast to equationality being preserved under naming and forgetting constants).

In Section \ref{sec: strong honest defs} we consider the relation of weak semi-equationality and distality in more detail.
We show that in an NIP theory, weak semi-equationality of a formula is equivalent to the existence of a \emph{one-sided} strong honest definition for it  (Theorem \ref{thm: w semieq and str hon def}). 
This is a simultaneous generalization of the existence of strong honest definitions in distal theories from \cite{chernikov2015externally} and the isolation property for the positive part of $\varphi$-types for equations (replacing a conjunction of finitely many instances of $\varphi$ by some formula $\theta$, see Fact \ref{fact: equation iff isolates positive}). 

In Section \ref{sec: val fields} we show that many theories of NIP valued fields with an infinite stable residue field, e.g.~$\ACVF$, are not weakly semi-equational (see Theorem \ref{thm: non weak semieq val fields} and Remark \ref{rem: val field more examples}). In Section \ref{sec: MS simple} we provide a sufficient criterion for when a formula is not a Boolean combination of weak semi-equations (generalizing  the criterion for equations from \cite{muller2017nonequational}). We then apply it to show that the partitioned formula $\psi(x_1,x_2; y_1,y_2) := \nu\left(x_{1}-y_{1}\right)<\nu\left(x_{2}-y_{2}\right)$
is not a Boolean combination of weak semi-equations via a detailed analysis of the behavior of indiscernible sequences. It remains open if the field $\mathbb{Q}_p$ is semi-equational (Problem \ref{prob: p-adics}).

In Section \ref{sec: exp by pred} we consider preservation of weak semi-equationality in expansions by naming a new predicate, partially adapting a result for NIP from \cite{chernikov2013externally}. Namely, we demonstrate in Theorem \ref{thm: semi-equational pairs} that if $\mathcal{M} \models T$ is distal, $A$ is a subset of $\mathcal{M}$ with a distal induced structure and the pair $\left(M,A \right)$ is almost model complete (i.e.~every formula in the pair is equivalent to a Boolean combination of formulas which only quantify existentially over the predicate, see Definition \ref{def: almost model compl}), then the pair $\left( \mathcal{M}, A \right)$ is weakly semi-equational. This implies in particular that dense pairs of $o$-minimal structures 
are weakly semi-equational (but not distal by \cite{hieronymi2017distal}).

\subsection*{Acknowledgements}
We thank the referee for many suggestions on improving the paper. We are grateful to Matthias Aschenbrenner, Gabe Conant, Allen Gehret, Anand Pillay, Sergei Starchenko, Erik Walsberg and Martin Ziegler for  some helpful comments and conversations. Both authors were partially supported by the NSF CAREER grant DMS-1651321, and Chernikov was additionally supported by a Simons fellowship. 

\section{Semi-equations and their basic properties}

Let $T$ be a complete theory in a language $\mathcal{L}$, we work inside a sufficiently saturated and homogeneous monster model $\mathbb{M} \models T$.
All sequences of elements are assumed to be small relative to the
saturation of $\mathbb{M}$, and we write $x,y,\ldots$ to denote finite tuples of variables. Given two linear orders $I,J$, $I+J$ denotes the linear order given by their sum (i.e.~$I < J$); and $(0)$ denotes a linear order with a single element. We write $\mathbb{N} = \{0, 1, \ldots\}$ and for $k \in \mathbb{N}$, $[k] = \{1, \ldots, k\}$. Given a partitioned formula $\varphi(x,y)$, we let $\varphi^*(y,x) := \varphi(x,y)$.

\subsection{Some basic properties of (weak) semi-equations}

\begin{remark}\label{rem: lin ord dont matter}
	By Ramsey and compactness we may equivalently replace $\omega$ by an arbitrary infinite linear order in Definition \ref{def: weak semi-eq}(1), and $\mathbb{Z}$ by  $I_{L}+\left(0\right)+I_{R} $ with $I_{L}, I_{R}$ arbitrary infinite linear orders in Definition \ref{def: weak semi-eq}(2). 
\end{remark}
\noindent By Ramsey, compactness, and taking automorphisms we also have:
\begin{prop}\label{prop: equiv to semi-eq}
A formula $\varphi(x,y)$ is a semi-equation if and only if there are no $b$, infinite linear orders $I_L, I_R$ and an indiscernible sequence $\left(a_{i}\right)_{i\in I_{L}+\left(0\right)+I_{R}}$
such that $\models\varphi\left(a_{i},b\right)$ for $i\in I_{L}+I_{R}$,
but $\not\models\varphi\left(a_{0},b\right)$.

\end{prop}

\begin{prop}\label{prop: Bool combs}
	\begin{enumerate}
	\item 	If $\varphi(x,y)$ is a semi-equation, then $\varphi(x,y)$ is a weak semi-equation. Hence every semi-equational theory is  weakly semi-equational.
		\item Semi-equations are closed under conjunctions and exchanging the roles of the variables.
		\item Weak semi-equations are closed under conjunctions and
disjunctions.
	\end{enumerate}
\end{prop}
\begin{proof}
	(1) Clear from definitions using  Proposition \ref{prop: equiv to semi-eq}.

	(2) Suppose $\varphi\left(x,y\right)\land\psi\left(x,y\right)$
is not a semi-equation. By Proposition \ref{prop: equiv to semi-eq},  there are $b$ and an indiscernible
sequence $\left(a_{i}\right)_{i\in\mathbb{Z}}$ such that $\models\varphi\left(a_{i},b\right)\land\psi\left(a_{i},b\right)\iff i\neq0$.
Either $\not\models\varphi\left(a_{0},b\right)$, in which case $\varphi\left(x,y\right)$
is not a semi-equation, or $\not\models\psi\left(a_{0},b\right)$,
in which case $\psi\left(x,y\right)$ is not a semi-equation. And $\varphi(x,y)$ is a semi-equation if and only if $\varphi^*(y,x) := \varphi(x,y)$ is a semi-equation by the symmetry of the definition.

(3) For conjunctions, same as the proof of (2), 
but with the stipulation that $\left(a_{i}\right)_{i\neq0}$ is $b$-indiscernible
added. Now suppose $\varphi\left(x,y\right)\lor\psi\left(x,y\right)$
is not a weak semi-equation. Then there is $b$ and an indiscernible sequence
$\left(a_{i}\right)_{i\in\mathbb{Z}}$ such that $\left(a_{i}\right)_{i\neq0}$
is $b$-indiscernible, and $\models\varphi\left(a_{i},b\right)\lor\psi\left(a_{i},b\right)\iff i\neq0$.
Either $\models\varphi\left(a_{1},b\right)$ or $\models\psi\left(a_{1},b\right)$,
and then, by $b$-indiscernibility, either $\models\varphi\left(a_{i},b\right)$
for all $i\neq0$ or $\models\psi\left(a_{i},b\right)$ for all $i\neq0$.
In the first case, $\varphi\left(x,y\right)$ is not a weak semi-equation,
and in the second case, $\psi\left(x,y\right)$ is not a weak semi-equation.
\end{proof}

\begin{remark}
(1) To see that neither property is closed under negation, note that $x=y$
is a semi-equation (hence also a weak semi-equation), but $x\neq y$
is not a weak semi-equation in the theory of infinite sets. 

(2) To see that
 semi-equations need not be closed under disjunction, note that
in a linear order, $x<y$ and $y<x$ are both  semi-equations,
but their disjunction is equivalent to $x\neq y$, which is not.
\end{remark}

\begin{problem}\label{prob: weak semieq symm}
	Are weak semi-equations closed under exchanging the roles of the variables, at least in NIP theories? Fact \ref{lem: lifting distality over a predicate} can be viewed as establishing this for the definition of distality, however the proof is not sufficiently local with respect to a formula witnessing failure of distality.
\end{problem}
\begin{prop}\label{prop: semieq reduct}
Assume we are given languages $\mathcal{ L} \subseteq \mathcal{ L}'$, a complete $\mathcal{L}$-theory $T$ and an $\mathcal{L}'$-theory $T'$ with $T \subseteq T'$, and a formula $\varphi(x,y) \in \mathcal{L}$.
	\begin{enumerate}
	\item The formula $\varphi(x,y)$ is a semi-equation in $T$ if and only if it is  in $T'$.
		\item 
If $\varphi\left(x,y\right)$ is a weak semi-equation in $T$,
then it is  a weak semi-equation in $T'$. 
	\end{enumerate}
\end{prop}
\begin{proof}
(1) Left to right is immediate from the definition (Proposition \ref{prop: equiv to semi-eq}). For the converse, assume that in some model of $T$ we can find an infinite sequence $\left(a_{i},b_{i}\right)_{i\in I}$ such
that for all $i,j \in I$, $\models\varphi\left(a_{i},b_{j}\right)\iff i\neq j$. By completeness of $T$, we can find arbitrarily long finite sequences with the same property in \emph{every} model of $T$, in particular in some model of $T'$. By compactness we can thus find an infinite sequence with the same property in a model of $T'$, demonstrating that $\varphi(x,y)$ is not a semi-equation in $T'$.

\noindent	(2) If $\varphi\left(x,y\right) \in \mathcal{L}$ is not a weak semi-equation in $T'$,
then (in a monster model of $T'$, and hence of $T$) there is $b$ and an $\mathcal{ L}'$-indiscernible $\left(a_{i}\right)_{i\in I_{L}+\left(0\right)+I_{R}}$
such that $\left(a_{i}\right)_{i\in I_{L}+I_{R}}$ is $\mathcal{ L}'$-indiscernible
over $b$ and $\models\varphi\left(a_{i},b\right)$ for $i\in I_{L}+I_{R}$,
but $\not\models\varphi\left(a_{0},b\right)$, for infinite linear
orders $I_{L},I_{R}$. Then, in particular, $\left(a_{i}\right)_{i\in I_{L}+(0) + I_{R}}$
is $\mathcal{ L}$-indiscernible, and $\left(a_{i}\right)_{i\in I_{L}+I_{R}}$
is $\mathcal{ L}$-indiscernible over $b$, so $\varphi\left(x,y\right)$
is a not a weak semi-equation in $T$.
\end{proof}

\begin{remark}
	The converse to Proposition \ref{prop: semieq reduct}(2) does not hold. Let $T' := \DLO$ be the theory of dense linear orders, and $T$ its reduct to $\mathcal{L} := \{ = \}$. Then the $\mathcal{L}$-formula $x \neq y$ is not a weak semi-equation in $T$ by inspection, but it is a weak semi-equation in $T'$ since it is equivalent to a disjunction of weak semi-equations $(x < y) \lor (x >y)$ (Proposition \ref{prop: Bool combs}).
\end{remark}

\begin{problem}
	Is weak semi-equationality of a theory preserved under reducts? This appear to be open already for equationality (see \cite[Question 3.10]{junker2000note}), and fails for semi-equationality (see Section \ref{sec: circ orders}).
\end{problem}

\begin{problem}
Is (weak) semi-equationality of theories invariant under bi-interpretability without parameters?
Equivalently, if $T$ is (weakly) semi-equational, does it follow that so is $T^{\eq}$?
\end{problem}


\subsection{Relationship to equations and NIP}\label{sec: rel to NIP etc}
We provide some evidence that semi-equationality can be naturally viewed as a generalization of equationality (in the sense of Srour) in stable theories to the NIP  context.
\begin{prop}\label{prop: semieq and stab}
	\begin{enumerate}
		\item  Weak semi-equations are NIP formulas, hence weakly semi-equational theories are NIP.
		\item Equations are semi-equations.
		\item  A formula is an equation if and only if it is both stable and
a  semi-equation.
\item In a stable theory, all weak semi-equations are equations. In particular, a stable theory is equational if and only if it is (weakly) semi-equational.
	\end{enumerate}
\end{prop}
\begin{proof}
	(1) If $\varphi\left(x,y\right)$ is not NIP, then there are an indiscernible
sequence $\left(a_{i}\right)_{i\in\mathbb{N}}$ and $b$ such that
$\models\varphi\left(a_{i},b\right) \iff i$ is even. For any finite
set of formulas $\Delta\left(x_{1}, \ldots, x_{n},y\right)$, by Ramsey's
theorem, there is an infinite $I\subseteq2\mathbb{N}$ on which the truth value of all formulas in $\Delta\left(a_{i_{1}}, \ldots, a_{i_{n}},b\right)$
is constant for all $i_{1}<\ldots < i_{n}\in I$. Thus, by letting
$a_{0}':=a_{i}$ for some sufficiently large odd $i$, we can find
an indiscernible sequence $\left(a_{i}'\right)_{i\in I_{L}+\left(0\right)+I_{R}}$
(using $I_{L}\sqcup I_{R}=I$, and $a_{i}'=a_{i}$ for $i\in I$)
for some infinite $I_{R}$ and arbitrarily large finite $I_{L}$,
such that $\left(a_{i}'\right)_{i\in I_{L}+I_{R}}$ is $\Delta$-indiscernible
over $b$. By compactness, it follows that $\varphi\left(x,y\right)$
is not a weak semi-equation.

\noindent (2) If $\varphi\left(x,y\right)$ is not a semi-equation, then
there is a sequence $\left(a_{i},b_{i}\right)_{i\in\mathbb{N}}$ such
that $\models\varphi\left(a_{i},b_{j}\right)\iff i\neq j$. In particular,
$\models\varphi\left(a_{i},b_{j}\right)$ for all $j<i$, and $\not\models\varphi\left(a_{i},b_{i}\right)$,
so this is a counterexample to the descending chain condition, and
$\varphi\left(x,y\right)$ is not an equation.

\noindent (3) If $\varphi\left(x,y\right)$ is not an equation, then by Ramsey and compactness there
is  an indiscernible sequence $\left(a_{i},b_{i}\right)_{i\in\mathbb{N}}$ such that $\models\varphi\left(a_{i},b_{j}\right)$
for all $j<i$, and $\not\models\varphi\left(a_{i},b_{i}\right)$.
If
$\varphi\left(a_{i},b_{j}\right)$ holds for $i<j$ then $\varphi\left(x,y\right)$ is not a semi-equation.
Otherwise, $\varphi\left(x,y\right)$ is not stable.

\noindent (4) If $\varphi\left(x,y\right)$ is not an equation, by Ramsey and compactness we can choose an indiscernible sequence $\left(a_{i},b_{i}\right)_{i\in\mathbb{Z}}$ such that $\models\varphi\left(a_{i},b_{j}\right)$ for all $j<i$,
and $\not\models\varphi\left(a_{i},b_{i}\right)$. The indiscernible sequence $(a_i,b_i)_{i \in \mathbb{Z}}$ is totally indiscernible by stability of $T$, hence we have $\models \varphi(a_i, b_0) \iff i \neq 0$, and also $(a_i : i \neq 0)$ is indiscernible over $b_0$. This shows that $\varphi\left(x,y\right)$ is not a weak semi-equation.
\end{proof}

\subsection{Weakly normal formulas, $\left(k,n\right)$-semi-equations and breadth.} \label{sec: quant semieq 1-based}

\begin{defn}(see \cite[Chapter 4, Definition 1.1]{pillay1996geometric}) A formula $\varphi\left(x,y\right)$ is 
 \emph{$k$-weakly normal} if for
every $b_{1}, \ldots, b_{k}\in \mathbb{M}^{y}$ such that $\models\exists x\,\varphi\left(x,b_{1}\right)\land \ldots \land\varphi\left(x,b_{k}\right)$,
there are some $i\neq j\in\left[k\right]$ such that $\models\forall x\,\varphi\left(x,b_{i}\right)\leftrightarrow\varphi\left(x,b_{j}\right)$.
It is \emph{weakly normal} if it is $k$-weakly normal
for some $k$ (by compactness this is equivalent to: an infinite
collection of pairwise distinct instances of $\varphi(x,y)$ must have empty intersection). 
\end{defn}
\noindent A formula $\varphi(x,y)$ is \emph{normal} in the sense of  
 \cite{pillay1983countable} if and only if it is $2$-weakly normal. Weakly normal formulas are a special kind  of equations characterizing ``linearity'' of forking in stable theories (see \cite[Chapter 4, Proposition 1.5 + Remark 1.8.4 + Lemma 1.9]{pillay1996geometric}):
\begin{fact}\label{fac: 1-based iff weakly normal}
	A stable theory $T$ is 1-based if and only if in $T$, every formula $\varphi(x,y) \in \mathcal{L}$, with $x,y$ arbitrary finite tuples of variables, is equivalent to a Boolean combination of finitely many weakly normal formulas $\psi_1(x,y), \ldots, \psi_n(x,y) \in \mathcal{L}$.
\end{fact}
\noindent We introduce some numeric parameters characterizing semi-equations, minimal values of which give rise to a generalization of weak normality.

\begin{defn}\label{def: (k,n)-semieq}
 For $k,n \in \mathbb{N}$, a formula $\varphi\left(x,y\right)$ is a \emph{$\left(k,n\right)$-semi-equation} if, for every $b_{1}, \ldots, b_{k}\in\mathbb{M}^y$,
if $\models\exists x \ \varphi\left(x,b_{1}\right)\land \ldots \land\varphi\left(x,b_{k}\right)$,
then for some pairwise distinct $i_{1}, \ldots, i_{n},j\in\left[k\right]$, $\models\forall x\,\left(\varphi\left(x,b_{i_{1}}\right)\land \ldots \land\varphi\left(x,b_{i_{n}}\right)\right)\rightarrow\varphi\left(x,b_{j}\right)$.
And $\varphi\left(x,y\right)$ is an \emph{$n$-semi-equation} if it is
a $\left(k,n\right)$-semi-equation for some $k$.
A theory $T$ is \emph{$n$-semi-equational} (respectively, \emph{$(k,n)$-semi-equational}) if every formula $\varphi(x,y) \in \mathcal{L}$, with $x,y$ arbitrary finite tuples of variables, is equivalent in $T$ to a Boolean combination of $n$-semi-equations (respectively, $(k,n)$-semi-equations) $\psi_1(x,y), \ldots, \psi_n(x,y) \in \mathcal{L}$.
\end{defn}

\begin{prop}\label{prop: k,n semieq basic}
\begin{enumerate}
	\item If $\varphi\left(x,y\right)$ is a $\left(k,n\right)$-semi-equation, then $n<k$, and $\varphi\left(x,y\right)$ is also
an $\left(\ell,m\right)$-semi-equation for any $\ell\geq k$
and $n\leq m<\ell$. If $\varphi\left(x,y\right)$ is an $n$-semi-equation, then it is also an $m$-semi-equation for every
$m\geq n$.
\item A formula is a semi-equation if and only if it is an
$n$-semi-equation for some $n$, if and only if it is an $(n,n-1)$-semi-equation for some $n$.
\end{enumerate}
	
\end{prop}
\begin{proof}
	(1) Clear from the definitions.
	
\noindent	(2) If $\varphi\left(x,y\right)$ is not a semi-equation,
let $\left(a_{i},b_{i}\right)_{i\in\mathbb{N}}$ be such that $\models\varphi\left(a_{i},b_{j}\right)\iff i\neq j$.
Then for any $\left(k,n\right)$ we have $\models\varphi\left(a_{0},b_{1}\right)\land \ldots \land\varphi\left(a_{0},b_{k}\right)$,
but for any pairwise distinct $i_{1}, \ldots, i_{n},j\in\left[k\right]$, $\models\varphi\left(a_{j},b_{i_{1}}\right)\land \ldots \land\varphi\left(a_{j},b_{i_{n}}\right)\land\neg\varphi\left(a_{j},b_{j}\right)$, hence $\varphi(x,y)$ is not a $(k,n)$-semi-equation.
Conversely, for any $k \in \mathbb{N}$, if $\varphi\left(x,y\right)$ is not a $\left(k,k-1\right)$-semi-equation, then there exist $b_{1}, \ldots, b_{k}$ such that for each
$j\in\left[k\right]$, there is $a_{j}$ such that $\models\varphi\left(a_{j},b_{i}\right)$
for $i\neq j$, but $\not\models\varphi\left(a_{j},b_{j}\right)$.
Hence if $\varphi\left(x,y\right)$ is not a $\left(k,k-1\right)$-semi-equation for any $k$, then by compactness $\varphi\left(x,y\right)$
is not a semi-equation. 
And if $\varphi\left(x,y\right)$ is not an $n$-semi-equation, then it is not an $\left(n+1,n\right)$-semi-equation by definition, so a formula that is not an $n$-semi-equation for any $n$ is also not a $\left(k,k-1\right)$-semi-equation for any $k$.
\end{proof}

We recall the notion of breadth from lattice theory.
\begin{defn}\cite[Section 2.4]{aschenbrenner2016vapnik}\label{def: breadth}
	Given a set $X$ and $d \in \mathbb{N}_{\geq 1}$, a family of subsets $\mathcal{F} \subseteq \mathcal{P}\left(X\right)$
has \emph{breadth} $d$ if any nonempty intersection of finitely many sets
in $\mathcal{F}$ is the intersection of at most $d$ of them, and $d$
is minimal with this property.
\end{defn}

\begin{prop}\label{prop: semieq iff fin breadth}
	A formula  $\varphi\left(x,y\right)$ is a $\left(k+1,k\right)$-semi-equation if and only if the family of sets $\mathcal{F}_{\varphi} := \left\{ \varphi\left(\mathbb{ M},b\right)\mid b\in\mathbb{ M}^{y}\right\} $
has breadth at most $k$.
In particular, $\varphi(x,y)$ is a semi-equation if and only if the family of sets $\mathcal{F}_{\varphi}$ has finite breadth.
\end{prop}
\begin{proof}
	The family of sets $\left\{ \varphi\left(\mathbb{M},b\right)\mid b\in\mathbb{ M}^{y}\right\} $
has breadth at most $k$ if and only if every finite consistent conjunction of instances
of $\varphi$ is implied by the conjunction of at most $k$ of those
instances. In particular this applies to consistent conjunctions of $(k+1)$ instances of $\varphi$, showing that if the breadth of $\mathcal{F}_{\varphi}$ is $\leq k$, then it is a $(k+1,k)$-semi-equation. Conversely, assume $\varphi(x,y)$ is a $(k+1,k)$-semi-equation. Given any consistent conjunction of $n>k$
instances of $\varphi$, any $(k+1)$ of them contain an instance implied
by the other $k$ instances. Removing this implied instance, we reduce to the case of a consistent conjunction of
$(n-1)$ instances, and after $(n-k)$ steps to a conjunction of $k$ instances of $\varphi$ implying all the other ones.
The ``in particular'' part is Proposition \ref{prop: k,n semieq basic}(2).
%
%
%
\end{proof}

\begin{expl}
Let $T$ be an NIP theory expanding a group, and let a formula $\varphi(x,y)$ be such that for every $b \in \mathbb{M}^y$, $\varphi( \mathbb{M}, b)$ is a subgroup. Then, by Baldwin-Saxl \cite{baldwin1976logical}, there exists $n \in \omega$ such that for all finite $B\subseteq \mathbb{M}^y$, there is $B_0\subseteq B$ with $\left|B_0\right|\leq n$ such that $\bigcap_{b\in B_0}\varphi\left(\mathbb{M},b\right)=\bigcap_{b\in B}\varphi\left(\mathbb{M},b\right)$. So $\varphi(x,y)$ is a semi-equation by Proposition \ref{prop: semieq iff fin breadth}.
\end{expl}
\begin{remark}
	If $\varphi(x,y)$ is stable with infinitely many distinct instances $\varphi(\mathbb{M},b), b \in \mathbb{M}^y$, then either $\varphi(x,y)$ is not a semi-equation, or $\neg \varphi(x,y)$ is not a semi-equation (combining \cite[Proposition 2.20]{aschenbrenner2016vapnik} and Proposition \ref{prop: semieq iff fin breadth}).
\end{remark}

The following suggests that $1$-semi-equationality can be viewed as a generalization of being $1$-based from stable to the NIP context.
\begin{prop}\label{prop: weak norm iff 1-semieq stab}
	A formula $\varphi\left(x,y\right)$ is weakly normal
if and only if it is stable and a $1$-semi-equation. Hence a stable theory is $1$-based if and only if it is $1$-semi-equational.
\end{prop}
\begin{proof}
	Clearly every $k$-weakly normal formula is a $\left(k,1\right)$-semi-equation and is also an equation, hence stable. 
Conversely, suppose that $\varphi\left(x,y\right)$ is a $\left(k,1\right)$-semi-equation and is stable, or just NSOP: there is some $\ell\in\omega$ such that there is no strictly increasing chain of sets of the form $\varphi\left(\mathbb{M},b_{0}\right) \subsetneq \ldots \subsetneq \varphi\left(\mathbb{M},b_{\ell}\right)$.
 We will
show that then $\varphi(x,y)$ is $k^{\ell}$-weakly normal.
Let $\left(b_{\eta}\right)_{\eta\in\left[k\right]^{\ell}}$ be such
that $\models\exists x\,\bigwedge_{\eta\in\left[k\right]^{\ell}}\varphi\left(x,b_{\eta}\right)$.
For $\sigma\in\left[k\right]^{\leq\ell}$, we will show by induction
on $m:=\ell-\left|\sigma\right|$ that there are pairwise distinct $\eta_{0}, \ldots, \eta_{m}\in\left[k\right]^{\ell}$
extending $\sigma$ (as sequences) such that $\varphi\left(\mathbb{ M},b_{\eta_{0}}\right)\subseteq\varphi\left(\mathbb{ M},b_{\eta_{1}}\right)\subseteq \ldots \subseteq\varphi\left(\mathbb{ M},b_{\eta_{m}}\right)$.
With $m=\ell$, so that $\sigma= \langle \rangle$ is the empty sequence, this implies by the choice of $\ell$ that there are $\eta \neq \eta'\in\left[k\right]^{\ell}$
such that $\varphi\left(\mathbb{ M},b_{\eta}\right)=\varphi\left(\mathbb{ M},b_{\eta'}\right)$,
as desired. The base case ($m=0$) is trivial, with $\eta_{0}=\sigma$.

Assume the claim holds for $m$, and let $\sigma\in\left[k\right]^{\ell-\left(m+1\right)}$.
For each $i\in\left[k\right]$, there exist pairwise distinct $\eta_{i,0}, \ldots, \eta_{i,m}\in\left[k\right]^{\ell}$
extending $\sigma^{\frown} i$ such that $\varphi\left(\mathbb{ M},b_{\eta_{i,0}}\right)\subseteq \ldots \subseteq\varphi\left(\mathbb{ M},b_{\eta_{i,m}}\right)$.
Among the sets $\left\{ \varphi\left(\mathbb{ M},b_{\eta_{i,0}}\right)\mid i\in\left[k\right]\right\} $,
one must be contained in another by $\left(k,1\right)$-semi-equationality. Say $\varphi\left(\mathbb{ M},b_{\eta_{j,0}}\right)\subseteq\varphi\left(\mathbb{ M},b_{\eta_{i,0}}\right)$
for some $i\neq j$. Then $\varphi\left(\mathbb{ M},b_{\eta_{j,0}}\right)\subseteq\varphi\left(\mathbb{ M},b_{\eta_{i,0}}\right)\subseteq  \varphi\left(\mathbb{ M},b_{\eta_{i,1}}\right)  \subseteq \ldots \subseteq\varphi\left(\mathbb{ M},b_{\eta_{i,m}}\right)$,
and $\eta_{j,0},\eta_{i,0}, \eta_{i,1}, \ldots ,\eta_{i,m}$ are pairwise distinct and extend $\sigma$, as desired.
The ``in particular'' part follows by Fact \ref{fac: 1-based iff weakly normal}.
\end{proof}

\begin{remark}
It is well known that the family of weakly normal formulas
is closed under conjunctions (but we could not find a direct reference). 
 While semi-equations are closed under conjunctions by Proposition \ref{prop: Bool combs}(2), this is not the case for the family of $1$-semi-equations.
Indeed, in  a dense linear order, the formulas $x < y_1$ and $x > y_2$ are $1$-semi-equations, but the formula $\varphi(x; y_1, y_2) := y_2 < x < y_1$ is not a $1$-semi-equation since we can have any number of intervals with a non-empty intersection, so that none of them is  contained in the other.

\end{remark}
We  observe a connection to another notion of ``linearity'' for NIP theories considered in \cite{basit2021zarankiewicz}, where various combinatorial results are proved for relations that are Boolean combinations of \emph{basic} relations.
The following is \cite[Definition 2]{basit2021zarankiewicz}, in the case of binary relations (using the equivalence in \cite[Proposition 2.8 + Remark 2.9]{basit2021zarankiewicz}).
\begin{defn}\label{defn: basic}
A binary relation $R \subseteq X \times Y$  is \emph{basic}	 if there exist a linear order $(S,<)$ and functions $f: X \to S, g: Y \to S $  for   such that for any $a \in X, b \in Y$, $(a,b) \in R \iff f(a) < g(b)$.
\end{defn}

\begin{fact}\cite[Claim 1 in the proof of Proposition 2.5]{guingona2013vc}\label{fac: GuingonaLask}
	Let $X$ be a set and $\mathcal{F} \subseteq \mathcal{P}(X)$ a family of subsets of $X$ such that there are no $A,B \in \mathcal{F}$ satisfying $A \cap B \neq \emptyset, B \setminus A \neq \emptyset$ and $B \setminus A \neq \emptyset$ simultaneously. Then there exists a linear order $<$ on $X$ so that every $A \in \mathcal{F}$ is a $<$-convex subset of $X$.
\end{fact}

\begin{prop}\label{prop: 1-semieq iff basic}
\begin{enumerate}
	\item Given a formula $\varphi(x,y) \in \mathcal{L}$, if the relation $R_{\varphi} := \{(a,b) \in \mathbb{M}^{x} \times \mathbb{M}^y : \models \varphi(a,b) \}$ is basic, then $\varphi(x,y)$ is a $(2,1)$-semi-equation.
	\item If $\varphi(x,y)$ is $(2,1)$-semi-equation, then $R_{\varphi} = R_1 \cap R_2$ for some (not necessarily definable) basic relations $R_1, R_2 \subseteq \mathbb{M}^{x} \times \mathbb{M}^y$.
	\end{enumerate}
\end{prop}
\begin{proof}
	(1) Let $(S,<), f,g$ be as in Definition \ref{defn: basic} for $R_{\varphi}$.
	Given any $b_1, b_2 \in \mathbb{M}^{y}$, the sets $\left\{x \in S : x < g(b_i) \right\}$ for $i \in \{1,2\}$ are initial segments of $S$. Say $g(b_1) \leq g(b_2)$. Then for any $a \in \mathbb{M}^x$, $f(a) < g(b_1) \Rightarrow f(a) < g(b_2)$, so $\varphi(\mathbb{M},b_1) \subseteq \varphi(\mathbb{M},b_2)$, and the other case is symmetric.
	
\noindent 	(2) If $\varphi(x,y)$ is a $(2,1)$-semi-equation, then the family $\mathcal{F}_{\varphi}$ of subsets of $\mathbb{M}^{x}$ satisfies the assumption in Fact \ref{fac: GuingonaLask}. Hence there exists a (not necessarily definable) linear ordering $<'$ of $\mathbb{M}^{x}$ so that for every $b \in \mathbb{M}^y$,  $\varphi(\mathbb{M},b)$ is $<'$-convex. Let $(S,<)$ be the Dedekind completion of $\left( \mathbb{M}^x, <' \right)$. Consider the functions $g_1, g_2: \mathbb{M}^{y} \to S$ so that  $g_1(b)$ is the infimum of $\varphi(\mathbb{M},b)$ in $S$, and $g_2(b)$ is the supremum of $\varphi(\mathbb{M},b)$ in $S$. Then $R_{\varphi} = \left\{ (a,b) \in \mathbb{M}^x \times \mathbb{M}^y  : g_1(b)\leq a \right\} \cap \left\{ (a,b) \in \mathbb{M}^x \times \mathbb{M}^y: a  \leq  g_2(b) \right\}$, and both of this relations are basic (see \cite[Remark 2.7]{basit2021zarankiewicz}).
	\end{proof}

\begin{remark}\label{rem: Zarank etc for 1-semieq}
(1) In view of Proposition \ref{prop: 1-semieq iff basic}(2), if $\varphi(x,y)$ is a Boolean combination of $(2,1)$-semi-equations, then by \cite[Theorem 2.17 + Remark 2.20]{basit2021zarankiewicz} the relation $R_{\varphi}$ satisfies an almost linear Zarankiewicz bound.  In particular, no infinite field can be defined in a $(2,1)$-semi-equational theory (see \cite[Corollary 5.11]{basit2021zarankiewicz} or \cite[Proposition 6.3]{walsberg2021notes} for a detailed explanation).

\noindent (2) If $\varphi(x,y)$ is a $(2,1)$-semi-equation, then $R_{\varphi}$ need not be basic. Indeed, the family of cosets of a subgroup is $(2,1)$-semi-equational. If it was basic, then its complement is also basic, hence $(2,1)$-semi-equational by the lemma above. But if the index of the subgroup is $\geq 3$, the family of complements of cosets is clearly not $(2,1)$-semi-equational.
\end{remark}

\begin{problem}
	If $\varphi(x,y)$ is a $(k,1)$-semi-equation for $k \geq 3$, is it still a Boolean combination of basic relations?
\end{problem}

\begin{problem}\label{prob: 1-semieq no field}
	Show that no infinite field is definable in a $1$-semi-equational theory.
\end{problem}

\begin{problem}
	Is every $1$-semi-equational theory rosy? (Note that dense trees are not $1$-semi-equational by Theorem \ref{thm: trees not 1-semieq}.)
\end{problem}

\section{Examples of semi-equational theories}\label{sec: semieq examples}
\subsection{O-minimal structures}\label{sec: ex omin}

All $o$-minimal theories (and more generally, ordered dp-minimal theories) are distal (see \cite{simon2013distal}), hence they are  weakly semi-equational by Remark \ref{rem: distal implies w semieq}. Semi-equationality appears more subtle. An $o$-minimal structure $\mathcal{M}$ is \emph{linear} if it has the \emph{CF property} in the sense of \cite{loveys1993linear}, i.e.~if every interpretable normal family of curves is of dimension at most $1$. This is a weakening of local modularity of the pregeometry induced by the algebraic closure, and  by the $o$-minimal trichotomy \cite{peterzil1998trichotomy} it is equivalent to no infinite field being definable in $\mathcal{M}$. We will only need the following fact about linear $o$-minimal structures from \cite{loveys1993linear}:
\begin{fact}\label{fac: CF}
	Let $T = \Th(\mathcal{M})$, with $\mathcal{M} = \left(M; <, +, \ldots \right)$ a linear $o$-minimal expansion of a group. Let $\mathcal{L} = (<,+,\ldots)$ be the language of $T$. A \emph{partial endomorphism} of $\mathcal{M}$ is a map $f: (-c,c) \to M$, for $c$ an element of $M$ or $\infty$, such that if $a,b,a+b$ are all in the domain, then $f(a+b) = f(a) + f(b)$.
Let $\mathcal{M}'$ be the reduct of $\mathcal{M}$ to the language $\mathcal{L}'$ consisting of $+,<$, constant symbols naming $\acl_{\mathcal{L}}(\emptyset)$, and for each $\mathcal{L}(\emptyset)$-definable partial endomorphism $f: (-c,c) \to M$ with $c \in \acl_{\mathcal{L}}(\emptyset)$ or $c = \infty$, a unary function symbol interpreted as $f$ on $(-c,c)$ and as $0$ outside of the domain of $f$. Let $T' := \Th_{\mathcal{L}'}(\mathcal{M}')$. 

\begin{enumerate}
	\item \cite[Proposition 4.2]{loveys1993linear} A subset of $M^n$ is $\emptyset$-definable in $\mathcal{M}$ if and only if it is $\emptyset$-definable in $\mathcal{M}'$.
	\item \cite[Corollary 6.3]{loveys1993linear} $T'$ admits quantifier elimination in the language $\mathcal{L}'$. 
\end{enumerate}

\end{fact}

\begin{prop}\label{prop: mod o-min 1-semieq}
	Let $T = \Th(\mathcal{M})$, with $\mathcal{M} = \left(M; <,  \ldots \right)$ an $o$-minimal structure.
	
	\begin{enumerate}
		\item If $T$ is an expansion of an ordered  group and linear, then $T$ is $(2,1)$-semi-equational.
		\item Conversely, if $T$ is $(2,1)$-semi-equational, then $T$ is linear.
	\end{enumerate}
\end{prop}

\begin{proof}

(1) Let $\mathcal{L} = (<,+,\ldots)$, $\mathcal{M}'$ and $\mathcal{L}'$ be as in Fact \ref{fac: CF}. By Fact \ref{fac: CF}(1) it suffices to show that $T' := \Th_{\mathcal{L}'} \left(\mathcal{M}'\right)$ is $(2,1)$-semi-equational. By Fact \ref{fac: CF}(2), it then  suffices to show that every atomic $\mathcal{L}'$-formula $\varphi(x,y)$, with $x,y$ arbitrary finite tuples of variables, is equivalent in $T'$ to a Boolean combination of $(2,1)$-semi-equations. 
 By the proof of Theorem 4.3 in \cite{anderson2021combinatorial}, every atomic $\mathcal{L}'$-formula $\varphi(x,y)$ is equivalent in $T'$ to a Boolean combination of atomic formulas of the form $f(x)  \square  g(y) + c$, where $\square \in \left\{<, = , > \right\}$,  $f: M^{|x|} \to M, g: M^{|y|} \to M$ are total multivariate $\mathcal{L}'(\emptyset)$-definable homomorphisms and $c \in \dcl_{\mathcal{L}'}(\emptyset)$. Every formula of this form clearly defines a basic relation on $M^{|x|} \times M^{|y|}$, hence is a $(2,1)$-semi-equation by Proposition \ref{prop: 1-semieq iff basic}(1).

\noindent (2) By the $o$-minimal trichotomy theorem (see \cite{peterzil1998trichotomy} and  Remark 2 after the statement of Theorem 1.7 there), if $\mathcal{M}$ is not linear, then it defines an infinite field. But then Remark \ref{rem: Zarank etc for 1-semieq}(1) implies that $T$ is not $(2,1)$-semi-equational.
%
\end{proof}

\begin{problem}
Is every $o$-minimal $1$-semi-equational structure linear? 
A positive answer would follow from a positive answer to Problem \ref{prob: 1-semieq no field}.
\end{problem}

\begin{problem}\label{prob: reals semieq}
	Which $o$-minimal theories are semi-equational? In particular, is $\Th(\mathbb{R}, +, \times)$ semi-equational? 
	\end{problem}

\subsection{Colored linear orders}\label{sec: lin ords}

	Given a linearly ordered set $(S,<)$, a binary relation $R \subseteq S^2$ 
is \emph{monotone} if $(x,y) \in R$, $x' \leq x$, and $y\leq y'$ 
implies $\left(x',y'\right) \in R$.

\begin{fact}	Let $\mathcal{M} = \left( M, <, \left(C_i \right)_{i \in I}, \left(R_j \right)_{j \in J} \right)$ be a linear order expanded by arbitrary unary ($C_i$) and monotone binary ($R_j$) relations. Then $\Th \left( \mathcal{M} \right)$ is $(2,1)$-semi-equational.
\end{fact}
\begin{proof}
Let $\mathcal{M}'$ be an expansion of $\mathcal{M}$ obtained by naming all $\mathcal{L}_{\mathcal{M}}(\emptyset)$-definable unary and monotone binary relations, then a subset of $M^n$ is $\emptyset$-definable in $\mathcal{M}$ if and only if it is $\emptyset$-definable in $\mathcal{M}'$, so it suffices to show that $T' := \Th \left(\mathcal{M}' \right)$ is $(2,1)$-semi-equational.
By \cite[Proposition 4.1]{simon2011dp}, $T'$ eliminates quantifiers. Note that if $R(x,y)$ is monotone, then it is a $(2,1)$-semi-equation (given any $b_1 \leq b_2 \in M$, for any $a \in M$ we have $\models R(a,b_1) \rightarrow R(a,b_2)$ by monotonicity, hence $R(M,b_1) \subseteq R(M,b_2)$). And any unary relation $C_i(x)$ is trivially a $(2,1)$-semi-equation, hence $T'$ is $(2,1)$-semi-equational.
\end{proof}

%

\subsection{Cyclic orders}\label{sec: circ orders}

 A cyclic order ${\circlearrowright}(x,y,z)$ (see e.g.~\cite[Section 5]{chernikov2012forking} or \cite{tran2017family}) is \emph{dense} if its underlying set is infinite and for every distinct $a,b$, there is $c$ such that ${\circlearrowright}\left(a,b,c\right)$, and $d$ such that ${\circlearrowright}\left(d,b,a\right)$. The theory $T_{{\circlearrowright}}$ of dense cyclic orders is complete and has quantifier elimination (see e.g. \cite[Proposition 3.7]{semieqExp}).

%

\begin{prop}\label{prop: cyclic order semieqs}
\begin{enumerate}
	\item The theory $T_{{\circlearrowright}}$ is not semi-equational.
\item The theory $T_{{\circlearrowright}}$ expanded with one constant
symbol $c$ is $(2,1)$-semi-equational.
\end{enumerate}

\end{prop}
\begin{proof}
(1) We show that $ \psi(x_1,x_2; y) := {\circlearrowright}\left(x_{1},x_{2};y\right)$
is not a Boolean combination of semi-equations. By quantifier elimination, the formulas ${\circlearrowright}\left(x_{1},x_{2};y\right)$
and ${\circlearrowright}\left(x_{2},x_{1};y\right)$ each isolate a complete
$3$-type (over $\emptyset$). Any Boolean combination of formulas that is equivalent to ${\circlearrowright}\left(x_{1},x_{2};y\right)$
must contain some formula $\varphi\left(x_{1},x_{2};y\right)$ that
is implied by ${\circlearrowright}\left(x_{1},x_{2};y\right)$ and is
inconsistent with ${\circlearrowright}\left(x_{2},x_{1};y\right)$,
or vice versa. Assume the former. Let $\left(c_{i}\right)_{i\in\mathbb{Z}}$ be 
such that $\models{\circlearrowright}\left(c_{k},c_{i},c_{j}\right)$
for $i<j<k$. Let $a_{1,i}=c_{2i}$, $a_{2,i}=c_{2i+2}$, and $b_{i}=c_{2i+1}$.
Then $\models{\circlearrowright}\left(a_{2,i},a_{1,i};b_{j}\right)\Leftrightarrow i=j$
and $\models{\circlearrowright}\left(a_{1,i},a_{2,i};b_{j}\right) \Leftrightarrow i\neq j$,
so $\models\varphi\left(a_{1,i},a_{2,i};b_{j}\right) \Leftrightarrow i\neq j$,
so $\varphi\left(x_{1},x_{2};y\right)$ is not a semi-equation.
If instead, $\varphi\left(x_{1},x_{2};y\right)$ is implied by ${\circlearrowright}\left(x_{2},x_{1};y\right)$
and inconsistent with ${\circlearrowright}\left(x_{1},x_{2};y\right)$,
we switch the roles of $x_{1}$ and $x_{2}$ to get the same
result.

\noindent(2) Let $<$ be defined by $x<y\Leftrightarrow{\circlearrowright}\left(x,y,c\right)$. Then 
$<$ is a dense linear order on the complement of $\left\{ c\right\} $,
so $x < y$ is a $\left(2,1\right)$-semi-equation. We have 
that ${\circlearrowright}\left(x,y,z\right)$ is equivalent to 
	$x<y<z\lor y<z<x\lor z<x<y\lor\left(z=c\land x<y\right)\lor\left(y=c\land z<x\right)\lor\left(x=c\land y<z\right)$.
Hence ${\circlearrowright}\left(x,y,z\right)$ is 
a Boolean combination of $\left(2,1\right)$-semi-equations (with $c$ as a parameter),
under any partition of the variables. By quantifier elimination, it follows
that every formula is a Boolean combination of $\left(2,1\right)$-semi-equations (using $c$ as a parameter).
\end{proof}
This example shows that a theory being semi-equational, or 1-semi-equational, is not preserved under forgetting constants (naming constants clearly preserves $k$-semi-equationality).  This is in contrast to equationality (\cite[Proposition 3.5]{junker2000note}) and distality (\cite[Corollary 2.9]{simon2013distal}), which are invariant under naming or forgetting constants. This is also an example of a distal, non-semi-equational theory.

\begin{problem}Is weak semi-equationality of theories preserved by forgetting constants?\end{problem}

\subsection{Ordered abelian groups} \label{sec: OAGS}
We consider ordered abelian groups, as structures in the language $\mathcal{L}_{\CH}$ introduced in \cite{cluckers2011quantifier}. Given an ordered abelian group $(G,+,<)$ and prime $p$, for $a \in G \setminus pG$ we let $G_p(a)$ be the largest convex subgroup of $G$ such that $a \notin G_p(a) +pG$, and for $a \in pG$ let $G_p(a) := \{0\}$. Let $\mathcal{S}_p := \{G_p(a) : a \in G\}$. Then the $\mathcal{L}_{\CH}$-structure $\bar{G}$ corresponding to $G$ consists of the main sort $\mathcal{G}$ for $G$, an auxiliary sort $\mathcal{S}_p$ for each $p$, along with countably many further auxiliary sorts and relations between them. A relative quantifier elimination result is obtained for such structures in \cite{cluckers2011quantifier}, to which we refer for the details (see also \cite[Section 3.2]{aschenbrenner2020distality} for a quick summary). 

\begin{prop}
Every ordered abelian group (either as a pure ordered abelian group, or the corresponding structure $\bar{G}$) with finite auxiliary sorts $\mathcal{S}_p$ for all prime $p$ is $1$-semi-equational (this includes Presburger arithmetic, and any ordered abelian group with a strongly dependent theory by \cite{chernikov2015groups, dolich2018characterization, farre2017strong, halevi2019strongly}). \end{prop}
\begin{proof}
	Since every auxiliary sort is finite and linearly ordered by a (definable) relation in $\mathcal{L}_{\CH}$, all auxiliary sorts are contained in $\dcl(\emptyset)$. Hence we only need to verify that every formula $\varphi(x,y)$ with $x,y$ tuples of the main sort $\mathcal{G}$ is a Boolean combination of $1$-semi-equations in the expansion with every element of every auxiliary sort named by a new constant symbol (countably many in total).
	As explained in \cite[Proposition 3.14]{aschenbrenner2020distality}, it then follows from the relative quantifier elimination that $\varphi(x,y)$ is equivalent to a Boolean combination of atomic formulas of the form $\pi_{\alpha}(f(x)) \diamond_{\alpha} \pi_{\alpha}(g(y)) + k_{\alpha}$, where $\diamond \in \left\{ =, <, \equiv_{m} \right\}$, $k \in \mathbb{Z}$, $\alpha$ is an element of an auxiliary sort, $f,g$ are $\mathbb{Z}$-linear functions on $\mathcal{G}$, $G_\alpha$ is a corresponding convex subgroup of $G$, $\pi_{\alpha}: G \to G/G_{\alpha}$ is the quotient map, $1_{\alpha}$ is the minimal positive element of $G/G_{\alpha}$ if it is discrete or $0 \in G/G_{\alpha}$ otherwise, and $k_{\alpha} = k \cdot 1_{\alpha}$ in $G/G_{\alpha}$, and for $g,h \in G/G_{\alpha}$ we have $g \equiv_m h$ if $g - h \in m \left( G/G_{\alpha} \right)$  (note that these relations on $G$ are expressible in the pure language of ordered abelian groups).
	
	It is straightforward from Definition \ref{def: (k,n)-semieq} that 
		if $\varphi(x,y)$ is a $(k,n)$-semi-equation  and $f(x), g(y)$ are $\emptyset$-definable functions, then the formula $\psi(x,y) := \varphi(f(x), g(y))$ is also a $(k,n)$-semi-equation.
		Using this (in an expansion of $\bar{G}$ naming $\pi_{\alpha}$, and the ordered group structure on $G/G_{\alpha}$ together with the constants for $k_{\alpha}$), we only have to show that the relations $x = y, x < y, x \in y + m\left(G/G_{\alpha} \right)$ on $G/G_{\alpha}$ are $(2,1)$-semi-equations, which is straightforward.
\end{proof}

\begin{problem}
	Is every ordered abelian group $1$-semi-equational, or at least (weakly) semi-equational?
\end{problem}

\subsection{Trees} \label{sec: trees}

In this section we use ``$\land$'' to denote ``meet'', and ``$\&$'' to denote conjunction.
 By a \emph{tree} we mean a meet-semilattice $(M, \land)$ with an associated partial order $\leq$ (defined by $x \leq y \iff x \land y = x$) so that all of its initial segments are linear orders.
An \emph{infinitely-branching dense tree} is a tree 
whose initial segments are dense linear orders and such that for each
element $x$, there are infinitely many elements any two of which
have meet $x$.

\begin{fact}\label{lem: QE in dense trees}(see e.g.~\cite[Lemma 3.14]{semieqExp} or \cite[Section 1]{mennuni2022weakly})
	The theory of infinitely-branching dense trees is complete and eliminates quantifiers
in the language $\left\{ \land\right\} $.
\end{fact}

\begin{lemma}\label{lem: trees 3-vars}
	In any tree $\mathcal{M} = (M, \land)$  with no additional structure,
if every formula of the form $\varphi\left(x;y_{1},y_{2}\right)$ with $x,y_1,y_2$ singletons is a Boolean
combination of semi-equations, then every formula is a Boolean
combination of semi-equations.
\end{lemma} 

\begin{proof}
By \cite[Corollary 4.6]{simon2011dp} (using that $x \leq y \iff x \land y = x$), in any tree $\mathcal{M}=(M, \land)$ we have: two tuples $\bar{a} = (a_i : i \in [n]),\bar{b} = \left( b_j: j \in [n] \right) \in M^{n}$ have
the same type if and only if $\left(a_{i},a_{j},a_{k}\right)$ and $\left(b_{i},b_{j},b_{k}\right)$
have the same type for every $i,j,k\in\left[n\right]$. Hence for any $\bar{a},\bar{b}$,
$\tp\left(\bar{a}\bar{b}\right)$ is implied by the set of formulas
satisfied by 3-element subtuples of $\bar{a}\bar{b}$. So if every
partitioned formula with 3 total free variables is a Boolean combination
of semi-equations, then $\tp\left(\bar{a}\bar{b}\right)$ is
implied by a Boolean combination of  semi-equations. It is enough
that every formula of the form $\varphi\left(x;y_{1},y_{2}\right)$
is a Boolean combination of semi-equations, because then by
symmetry, every formula of the form $\varphi\left(x_{1},x_{2};y\right)$
is as well, and every partitioned formula with one of the parts empty
(i.e.~$\varphi\left(;y_{1},y_{2},y_{3}\right)$ or $\varphi\left(x_{1},x_{2},x_{3};\right)$)
is automatically a semi-equation.	
\end{proof}

\begin{theorem}\label{thm: trees are semieq}
	The theory of infinitely-branching dense trees is semi-equational.
\end{theorem}
\begin{proof}
	Let $\mathcal{ M} = (M, \land)$ be an infinitely-branching dense tree. By Lemma \ref{lem: trees 3-vars}, it is
enough to check that every formula $\varphi\left(x;y_{1},y_{2}\right)$
is a Boolean combination of semi-equations, and, by Fact \ref{lem: QE in dense trees}, it is enough to check this for positive atomic formulas
$\varphi\left(x;y_{1},y_{2}\right)$. Using the fact that $\land$
is associative, commutative, and idempotent, each such formula is equivalent
to a formula of the form $\bigwedge A=\bigwedge B$ for non-empty
$A,B\subseteq\left\{ x,y_{1},y_{2}\right\} $.  By a direct case analysis (see \cite[Theorem 3.16]{semieqExp} for the details) every such formula is either a tautology, or does not mention $x$, or an equality between two variables, or is equivalent to a Boolean combination of the following formulas (possibly replacing $y_2$ by $y_1$):

\noindent(1) $x=x\land y_{1}$, i.e.~$x\leq y_{1}$ --- a semi-equation: given $\left(a_{i},b_{i}\right)_{i\in\mathbb{Z}}$
such that $\models a_{i}\leq b_{j}\iff i\neq j$, $a_{i}\leq b_{0}$
for $i\neq0$, so $\left(a_{i}\right)_{i\neq0}$ forms a chain. This
is not consistent with $a_{1}\leq b_{2}$, $a_{2}\leq b_{1}$, $a_{1}\not\leq b_{1}$,
$a_{2}\not\leq b_{2}$. 

\noindent(2) $x \land y_1 \land y_2 = y_1 \land y_2$, i.e.~$x \geq y_1 \land y_2$ --- a semi-equation for the same
reason.

\noindent(3) $x\land y_{1}=x\land y_{2}$ --- a negated semi-equation: given $\left(a_{i},b_{i},b_{i}'\right)_{i\in\mathbb{Z}}$
such that $\models a_{i}\land b_{j}=a_{i}\land b_{j}'\iff i=j$,
for every $i\neq0$ we have: either $a_{i}\land b_{0}>a_{0}\land b_{0}$ or $a_{i}\land b_{0}'>a_{0}\land b_{0}$.
By pigeonhole, there are $i_{1}\neq i_{2}$ such that the same case
holds for both. Without loss of generality, $a_{1}\land b_{0}>a_{0}\land b_{0}$ and $a_{2}\land b_{0}>a_{0}\land b_{0}$.
But then $a_{1}\land a_{2}>a_{0}\land b_{0}=a_0\land a_1$, so $a_{1}$ and $a_{2}$ meet strictly closer to each other than to $a_{0}$. But, since $a_1\land b_1\leq a_0\land b_1$ and $a_1\land b_1\leq a_2\wedge b_1$, it also must be true that $a_{0}\land a_{2}\geq a_1\land b_1=a_1\land a_0$, so $a_{0}$
and $a_{2}$ meet at least as closely to each other as to $a_{1}$.
These are inconsistent.

\noindent(4) $x\land y_{1}=x\land y_{1}\land y_{2}$ (i.e.~$x\land y_{1}\leq y_{2}$) --- a negated semi-equation: given $\left(a_{i},b_{i},b_{i}'\right)_{i\in\mathbb{Z}}$
such that $\models a_{i}\land b_{j}\leq b_{j}'\iff i=j$, in particular $a_{0}\land b_{0}\leq b_{0}'$ and $a_i\land b_0\not\leq b_{0}'$ for $i\neq0$. Since the initial segment below $b_0$ is totally ordered, it follows that $a_0\land b_0 < a_i \land b_0$ for $i\neq0$. $a_1\land a_2\geq \left(a_1\land b_0\right)\land\left(a_2\land b_0\right)>a_0\land b_0=a_0\land a_1$. That is, $a_1$ and $a_2$ meet strictly closer together with each other than with $a_0$. But, by switching the roles of the indices $0$ and $2$ in that argument, $a_0$ and $a_1$ must meet strictly closer together with each other than with $a_2$ as well, a contradiction.
\end{proof}


\begin{theorem}\label{thm: trees not 1-semieq}
	In an infinitely-branching dense tree $\mathcal{M}=(M, \land)$, the formula $x<y$
is not a Boolean combination of $1$-semi-equations (without
parameters).
\end{theorem}

\begin{proof}
	By quantifier elimination, there are $4$ complete $2$-types over $\emptyset$ axiomatized by $\left\{ x=y,\,x>y,\,x<y,\,x\perp y\right\} $,
where $\perp$ denotes incomparable elements. Thus, up to equivalence,
there are only $16$ formulas $\varphi\left(x,y\right)$ with $x,y$ singletons without parameters. By a direct case analysis (see \cite[Theorem 3.17]{semieqExp} for the details) the only $1$-semi-equations among them are $x\neq x$, $x=x$,  $x=y$, $x>y$, $x\geq y$.  
None of them separate $x<y$ from $x\perp y$, so any Boolean combination
of $1$-semi-equations implied by $x<y$ must also be implied by $x\perp y$, so 
$x<y$ is not equivalent to a Boolean combination of $1$-semi-equations.
\end{proof}

\begin{cor}\label{cor: tree const not 1-semieq}
In any expansion of an infinitely-branching dense tree $\mathcal{M}=(M, \land)$ by naming constants, the formula $x<y$ is not a Boolean combination of $1$-semi-equations.
\end{cor}

\begin{proof}
Suppose $x<y$ is equivalent to a Boolean combination of $1$-semi-equations with parameters $c=\left(c_1, \ldots, c_n\right)$. Say $x<y\iff \Phi\left(\varphi_1\left(x,y,c\right), \ldots, \varphi_k\left(x,y,c\right)\right)$, where $\Phi$ is a Boolean formula in $k$ variables, and $\varphi_1\left(x,y,c\right), \ldots, \varphi_k\left(x,y,c\right)$ are $1$-semi-equations. Let $d$ be an element such that $d\perp \bigwedge_{i\leq n}c_i$.
 For each $i$, let 
$\psi_i\left(x,y\right)$ be the formula $\exists z \left( \tp(z) =  \tp\left(c\right) \& \left(x\wedge y\perp\bigwedge_{i\leq n}z_i\right)  \& \, \varphi_i\left(x,y,z\right) \right)$. 
As $\tp\left(c\right)$ is isolated by quantifier elimination, this is indeed a first-order formula. For $a,b>d$, if $\models\varphi_i\left(a,b,c\right)$, then $\models\psi_i\left(a,b\right)$. By quantifier elimination and \cite[Lemma 4.4]{simon2011dp}, the converse also holds. Thus, for $a,b>d$, $
	\models a<b\iff\models\Phi\left(\varphi_1\left(a,b,c\right), \ldots, \varphi_k\left(a,b,c\right)\right)$$\iff\models\Phi\left(\psi_1\left(a,b\right), \ldots,\psi_k\left(a,b\right)\right)$.
Since all singletons have the same type, it follows that this holds for all $a,b$. It thus remains to show that each $\psi_i\left(x,y\right)$ is a $1$-semi-equation, contradicting Theorem \ref{thm: trees not 1-semieq}.
If this were not the case for some $i\leq k$, then there would be $\left(b_j\right)_{j\in\mathbb N}$ and $a$ such that $\models\psi_i\left(a,b_j\right)$ for all $j\in\mathbb N$, but such that for every $j\neq\ell\in\mathbb N$, there is $a_{j,\ell}$ such that $\models\psi_i\left(a_{j,\ell},b_j\right)$ but $\not\models\psi_i\left(a_{j,\ell},b_\ell\right)$. But, again because all singletons have the same type, and every finite set of elements has a lower bound, it is consistent that furthermore all of these elements are above $d$. But then this would also provide a counterexample to $\varphi_i\left(x,y\right)$ being a $1$-semi-equation.
\end{proof}

\begin{remark}
	Since $x>y$ is a $\left(2,1\right)$-semi-equation
and $x<y$ is not, this shows that being an $\left(n,k\right)$-semi-equation for fixed $n,k$ (or even being a Boolean combination of them) is not preserved under exchanging the roles of the variables (while being a semi-equation is preserved).
\end{remark}

\begin{remark}
	Note also that every tree admits an expansion in which
$x<y$ is a Boolean combination of $\left(2,1\right)$-semi-equations.
In a tree, let $\leq_{\lex}$ be a linear order refining $\leq$ such
that for $a,b,b'$ such that $a\perp b$ and $b\wedge b'>b\wedge a$,
$a\leq_{\lex}b\iff a\leq_{\lex}b'$. Then let $\leq_{\revlex}$ be
given by $x\leq_{\revlex}y \ :\iff  \ x\leq y\lor\left(x\perp y\&y\leq_{\lex}x\right)$.
Then $\leq_{\revlex}$ satisfies the same conditions as $\leq_{\lex}$
(so both are $\left(2,1\right)$-semi-equations as both are
linear orders), and $x\leq y\iff x\leq_{\lex}y\&x\leq_{\revlex}y$.
\end{remark}

\begin{problem}\label{prob: trees semieq}
	Is every theory of trees semi-equational? Is every theory of trees (expanded by constants) not $1$-semi-equational?
\end{problem}

\section{Weak semi-equations and strong honest
definitions} \label{sec: strong honest defs}

In this section we discuss how (weak) semi-equationality naturally  generalizes both distality and equationality.

\begin{defn}
	Given a formula $\varphi\left(x,y\right)\in\mathcal{L}$ and a type $p$, we  
denote by  
$p_{\varphi}^{+} := \left\{ \varphi\left(x,b\right):\varphi\left(x,b\right)\in p\right\} $ 
 the positive $\varphi$-part of the type $p$.
\end{defn}

	Given small sets $A,B,C \subseteq \mathbb{M}$, let $A\ind_{C}^{u}B$ denote that
$\tp\left(A/BC\right)$ is finitely satisfiable in $C$.
 We recall
the following characterization of equations from \cite[Lemma 2.4]{martin2020equational}, which in turn is a variant of \cite[Theorem 2.5]{srour1988notion1}. Note that Fact \ref{fact: equation iff isolates positive}(3) below is equivalent to 
 \cite[Lemma 2.4(3)]{martin2020equational} since in stable theories non-forking is symmetric and equivalent to finite
satisfiability over models. Existence of $k$ in Fact \ref{fact: equation iff isolates positive}(2)
is not stated explicitly in \cite[Lemma 2.4(2)]{martin2020equational}, but is immediate from the proof.
\begin{fact}
\label{fact: equation iff isolates positive} Given a formula $\varphi\left(x,y\right)$
in a stable theory $T$, the following are equivalent:
\begin{enumerate}
\item  $\varphi\left(x,y\right)$ is an equation (equivalently, $\varphi^*(y,x) := \varphi(x,y)$ is an equation);
\item there is some $k\in\mathbb{N}$ such that for any $a \in \mathbb{M}^x$ and small $B \subseteq \mathbb{M}^y$, there is a subset $B_{0}$ of $B$ of size at most $k$
such that $\tp_{\varphi}^{+}\left(a/B_{0}\right)\vdash\tp_{\varphi}^{+}\left(a/B\right)$.
\end{enumerate}
\end{fact}
On the other hand, we recall one of the standard characterizations of distality (see e.g.~\cite[Corollary 1.11]{aschenbrenner2020distality}), which we use as a definition here:
\begin{defn}
	A theory is \emph{distal} if and only if every formula $\varphi\left(x,y\right)$ is \emph{distal}, that is, for any $I_{L}$
and $I_{R}$ infinite linear orders,  $b \in \mathbb{M}^y$ and
indiscernible sequence $\left(a_{i}\right)_{i\in I_{L}+\left(0\right)+I_{R}}$ with $a_i \in \mathbb{M}^x$
such that $\left(a_{i}\right)_{i\in I_{L}+I_{R}}$ is indiscernible
over $b$, $\models\varphi\left(a_{0},b\right)\iff\models\varphi\left(a_{i},b\right)$
for $i\in I_{L}+I_{R}$. 
\end{defn}
\noindent There is a straightforward relationship between weak semi-equationality and distality:
\begin{remark}\label{rem: distal implies w semieq}
	A formula $\varphi\left(x,y\right)$ is distal if and only if 
both $\varphi\left(x,y\right)$ and $\neg\varphi\left(x,y\right)$
are weak semi-equations. In particular, every distal theory is weakly semi-equational.
\end{remark} 
\begin{problem}
	Is there an NIP theory without a (weakly) semi-equational expansion? We note that while the theory $\ACF_p$ for $p>0$ is known not to have a distal expansion \cite{chernikov2015regularity}, it is equational, and hence semi-equational.
\end{problem}

An NIP theory is distal if and only if every formula admits a \emph{strong honest definition}:
\begin{fact}\cite[Theorem 21]{chernikov2015externally}\label{fac: str honest defs}
A theory $T$ is distal if and only if for every formula $\varphi(x,y)$  
there is a formula $\theta\left(x; y_1, \ldots, y_k\right)$, called a \emph{strong honest
definition} for $\varphi\left(x,y\right)$, such that for any finite
set $C \subseteq \mathbb{M}^y$ ($\left|C\right|\geq2$) and $a\in\mathbb{ M}^{x}$,
there is $b\in C^{k}$ such that $\models\theta\left(a;b\right)$ 
and $\theta\left(x;b\right)\vdash \tp_{\varphi}\left(a/C\right)$.
\end{fact}
We  now show that in an NIP theory, weak semi-equationality is equivalent to the existence of \emph{one-sided} strong honest definitions, which is also a generalization of Fact \ref{fact: equation iff isolates positive} (replacing a conjunction of finitely many instances of $\varphi$ by some formula $\theta$). We will need the following $(p,k)$-theorem of Matou\v{s}ek from combinatorics:
\begin{fact}\cite{matousek2004bounded}\label{fac: p,k theorem}
Let $\mathcal{F}$ be a family of subsets of some set $X$. Assume that the VC co-dimension of $\mathcal{F}$ is bounded by $k$. Then for every $p \geq k \in \mathbb{N}$, there is $N \in \mathbb{N}$ such that: for every finite subfamily $\mathcal{G} \subseteq \mathcal{F}$, if $\mathcal{G}$ has the $(p,k)$-property, meaning that among any $p$ subsets of $\mathcal{G}$ some $k$ intersect, then there is a subset of $X$ of size $N$ intersecting all sets in $\mathcal{G}$.
 	
\end{fact}
\begin{theorem}\label{thm: w semieq and str hon def}
\label{thm: semi-equation iff strong honest def for positive}Let
$T$ be NIP, and let $\varphi\left(x,y\right)$ be a formula. The following
are equivalent:
\begin{enumerate}
\item The formula $\varphi^*\left(y,x\right) := \varphi(x,y)$ is a weak semi-equation.
\item For every small $B \subseteq \mathbb{M}^{y}$ and $a \in \mathbb{M}^x$ with $\models \varphi(a,b)$ for all $b \in B$ there are $\theta(x;y_1, \ldots, y_k), c \in \left( \mathbb{M}^y \right)^k$ such that $c \ind^u_{B} a$, $\models \theta(a,c)$ and $\theta(x,c) \vdash \tp_{\varphi}^+ \left(a/B \right)$.
\item There is some formula $\theta\left(x;y_{1},\ldots,y_{k}\right)$ and number $N$ such
that for any finite set $B \subseteq \mathbb{M}^y$ with $\left|B\right|\geq2$ and $a \in \mathbb{M}^x$,
there is some $B_0 \subseteq B$ with $|B_0| \leq N$
 such that $\tp_{\theta}^+\left(a/B_0\right) \vdash \tp_{\varphi}^{+}\left(a/B\right)$.

\end{enumerate}
\end{theorem}

\begin{proof}

(1) implies (2). We follow closely the proof of \cite[Proposition 19]{chernikov2015externally}. Assume that $a, B$ are such that $\models \varphi(a,b)$ for all $b \in B$. Let $\mathcal{M} \preceq \mathbb{M}$ contain $a,B$, and $\left(\mathcal{M}',B' \right) \succ \left(\mathcal{M},B \right)$ be a $\kappa := \left| M \right|^+$-saturated elementary extension (with $B$ named by a new predicate). We may assume $\mathcal{M}' \prec \mathbb{M}$ is a small submodel. Take $p(x) := \tp \left(a / B' \right)$.

\begin{claim}
Assume that $q\left(y\right)\in S_{y}\left(B'\right)$
is a type finitely satisfiable in $B$. Then $p\left(x\right)\cup q\left(y\right)\vdash\varphi\left(x,y\right)$.
\end{claim}
\begin{proof}
Let $\hat{q} \in S_y(\mathbb{M})$ be an arbitrary global type extending $q$ and finitely satisfiable in $B$, and form the Morley product $\hat{q}^{(\omega)}(y_1, y_2, \ldots) := \bigotimes_{i \in \mathbb{N}} q(y_i) \in S_{(y_1, y_2, \ldots)}(\mathbb{M})$, also finitely satisfiable in $B$. For any set $C \subseteq \mathbb{M}$, we let $q|_{C} := \hat{q} \restriction_{C}$ (respectively, $q^{(\omega)}|_{C} := \hat{q}^{(\omega)} \restriction_{C}$) be the restriction of $\hat{q}$ (respectively, of $\hat{q}^{(\omega)}$) to formulas with parameters in $C$.
%
  As $T$ is NIP, by \cite[Lemma 5]{chernikov2015externally} there is some $D$ with $B \subseteq D \subseteq B', |D| < \kappa$ such that for any two realizations $I, I' \subseteq B'$ of $q^{(\omega)}|_{D}$ we have $aI \equiv_{D} aI'$. Fix some $I \models q^{(\omega)}|_{D}$ in $B'$ (exists by saturation of $(\mathcal{M}',B')$ and finite satisfiability of $q^{(\omega)}|_{\mathbb{M}}$ in $B$) and $J \models q^{(\omega)}|_{\mathbb{M}}$ (in some larger monster model $\mathbb{M}' \succ \mathbb{M}$).
 
 We claim that $I+J$ is indiscernible over $aB$. Indeed, as $q^{(\omega)}|_{\mathbb{M}}$ is finitely satisfiable in $B$, by compactness and saturation of $\left(\mathcal{M}', B' \right)$ there is some $J' \models q^{(\omega)}|_{a D I}$ in $B'$. If $I+J$ is not $aB$-indiscernible, then $I' + J'$ is not $aB$-indiscernible for some finite subsequence $I'$  of $I$. As by construction both $I' + J'$ and $J'$ realize $q^{(\omega)}|_{D}$ in $B'$, it follows by the choice of $D$ that $J'$ is not indiscernible over $aB$ --- contradicting the choice of $J'$.
 
 Now let $b^* \in \mathbb{M}$ be any realization of $q$, then the sequence $I + (b^*) + J$ is Morley in $q|_{\mathbb{M}}$ over $B$, hence indiscernible (even over $B$). And $I+J$ is indiscernible over $a$ (even over $aB$) by the previous paragraph. Note also that $\models \varphi(a,b)$ for every $b \in B'$ (by assumption we had $\models \varphi(a,b)$ for all $b \in B$, but $a \in \mathcal{M}$ and $\left(\mathcal{M}',B' \right) \succ \left(\mathcal{M},B \right)$). Hence $\models \varphi(a,b)$ for every $b \in I +J$. And since $\varphi^*(y,x)$ is a weak semi-equation, this implies $\models \varphi(a, b^*)$. That is, for any $a \models p$ and $b^* \models q$, we have $\models \varphi(a,b^*)$, as wanted.
 \end{proof}
Now let $S'$ be the set of types over $B'$ finitely satisfiable in $B$, then $S'$ is a closed subset of $S_y(B')$. By the claim, for every $q \in S'$ we have $p(x) \cup q(y) \vdash \varphi(x,y)$, hence by compactness $\theta_q(x) \cup \psi_q(y) \vdash \varphi(x,y)$ for some formulas $\theta_q(X) \in p, \psi_q(y) \in q$. As $\left\{\psi_q(y) : q \in S' \right\}$ is a covering of the closed set $S'$, it has a finite sub-covering $\left\{\psi_{q_k} : k \in K\right\}$. Let $\theta(x) := \bigwedge_{k \in K} \theta_{q_k} (x)\in p(x)$. 
As in particular $\tp(b/B) \in S'$ for every $b \in B$, we thus have $\theta(x) \in \mathcal{L}(B')$ (and $B' \ind^u_{B} a$), $\models \theta(a)$ and $\theta(x) \vdash \tp^+_{\varphi}(a/B)$.

\noindent(2) implies (3). Let $a,B$ be given. We either have that $\models \neg \varphi(a,b)$ holds for all $b \in B$, in which case $\tp_{\theta}^+\left(a/B_0\right) = \tp_{\varphi}^{+}\left(a/B\right) = \emptyset$, and $\emptyset \vdash \emptyset$ trivially. Or we replace $B$ by $\left\{ b \in B : \  \models \varphi(a,b) \right\}$, and follow the proof of (1) implies (2) in \cite[Theorem 21]{chernikov2015externally}.

We provide the details. By (2), given small $B\subseteq\mathbb{M}^y$ and $a\in\mathbb{M}^x$ such that $\models\varphi\left(a,b\right)$ for all $b\in B$, there exist $\theta\left(x;y_1,\ldots,y_\ell\right)$ and $c\in\left(\mathbb{M}^y\right)^\ell$ such that $c\ind_B^u a$, $\models\theta\left(a,c\right)$, and $\theta\left(x,c\right)\vdash \tp_\varphi^+\left(a/B\right)$. Then given  any finite $B_0\subseteq B$, there is $d \in B^\ell$ such that $\tp_{\varphi}(d/a B_0) = \tp_{\varphi}(c/a B_0)$, so in particular $\models\theta\left(a,d\right)$ and $\theta\left(x,d\right)\vdash \tp_\varphi^+\left(a/B_0\right)$.

Now fix an arbitrary function $f: \mathcal{L} \to \mathbb{N}$ and let $n_\theta := f \left( \theta\left(x;y_1,\ldots,y_\ell\right) \right)$
for every partitioned formula $\theta \in \mathcal{L}$ with $x$ same as before and $\ell $ arbitrary. Let $T_f$ be a theory in the language $\mathcal{L} \cup \{P(x),a \}$ with $P$ a new unary predicate and $a$ a new constant symbol, so that $T_f$ expands $T$ with the following axioms: $\forall x\in P\,\varphi\left(a,x\right)$ and, for every formula $\theta\left(x;y_1,\ldots,y_\ell\right) \in \mathcal{L}$, an axiom $\exists b_1,\ldots,b_{n_\theta}\in P\,\forall c\in P^\ell\,\left(\neg\theta\left(a,c\right)\right) \lor \exists x\,\left(\theta\left(x,c\right) \land \bigvee_{i\leq n_\theta}\neg\varphi\left(a,b_i\right)\right)$. By the previous paragraph, the theory $T_f$ is inconsistent. By compactness, there is a finite inconsistent subset of $T_f$ only requiring finitely many of these formulas $\theta_1,\ldots,\theta_k$.

Thus there are finitely many formulas $\theta_1\left(x;y_1,\ldots,y_{\ell_1}\right),\ldots,\theta_k\left(x;y_1,\ldots,y_{\ell_k}\right) \in \mathcal{L}$ such that: given $B\subseteq\mathbb{M}^y$ and $a\in\mathbb{M}^x$ such that $\models\varphi\left(a,b\right)$ for all $b\in B$, there is $i\leq k$ such that for all $B_0\subseteq B$ with $\left|B_0\right|\leq n_{\theta_i}$, there is $c\in B^{\ell_i}$ such that $\models\theta_i\left(a;c\right)$ and $\theta_i\left(x;c\right)\vdash \tp_\varphi^+\left(a/B_0\right)$.

For each formula $\theta\left(x;y_1,\ldots,y_\ell\right) \in \mathcal{L}$, let $\rho_\theta\left(x,y;z\right):=\theta\left(x;z\right)\land\forall w\,\theta\left(w;z\right)\rightarrow\varphi\left(w,y\right)$, and $n_\theta:=\text{VC}\left(\rho_\theta\right)+1$ in the above argument (where $\VC$ is the VC-dimension), and let $\theta_1, \ldots, \theta_k$ be as given by the previous paragraph for this choice of the $n_{\theta}$'s. Then for an arbitrary $a\in\mathbb{M}^x$ and finite $B\subseteq\mathbb{M}^y$ such that $\models\varphi\left(a;b\right)$ for $b\in B$, there is $i_{a,B}\leq k$ such that: for all $B_0\subseteq B$ with $\left|B_0\right|\leq n_{\theta_{i_{a,B}}}$, there is $c\in B^{\ell_{i_{a,B}}}$ with $\models\theta_{i_{a,B}}\left(a;c\right)$ and $\theta_{i_{a,B}}\left(x;c\right)\vdash \tp_\varphi^+\left(a/B_0\right)$. For $b\in B$, consider the set 
	$$F_{a,B}^b:=\left\{c\in B^{\ell_{i_{a,B}}}: \,\models\theta_{i_{a,B}}\left(a;c\right),\,\theta_{i_{a,B}}\left(x;c\right)\vdash\varphi\left(x;b\right)\right\}.$$
Note that $F_{a,B}^b=\left\{c\in B^{\ell_{i_{a,B}}} : \,\models\rho_{\theta_{i_{a,B}}}\left(a,b;c\right)\right\}$, 
 and let $\mathcal{F}_{a,B}:=\left\{F_{a,B}^b : b\in B\right\}$. By Fact \ref{fac: p,k theorem} applied to $\mathcal{F}_{a,B}$, with $p=k=n_{\theta_{i_{a,B}}}$, there is $N_{i_{a,B}}$ (depending on $i_{a,B}$ but not otherwise depending on $a,B$) such that if every $n_{\theta_{i_{a,B}}}$ sets from $\mathcal{F}_{a,B}$ intersect, then there is $B_0\subseteq B^{\ell_{i_{a,B}}}$ with $\left|B_0\right|\leq N_{i_{a,B}}$ intersecting all sets from $\mathcal{F}_{a,B}$. Furthermore, by choice of $i_{a,B}$, the condition that every $n_{\theta_{i_{a,B}}}$ sets from $\mathcal{F}_{a,B}$ intersect holds. And there are only $k$ many possible values of $i_{a,B}$, so we let $N:=\max_{1 \leq i \leq k} N_i$.

We thus found $N\in\mathbb N$ such that: for all $a\in\mathbb{M}^x$ and finite $B\subseteq\mathbb{M}^y$ with $\models\varphi\left(a;b\right)$ for all $b\in B$, there is $i_{a,B}\leq k$ and $B_1\subseteq B^{\ell_{i_{a,B}}}$ with $\left|B_1\right|\leq N$ intersecting all sets from $\mathcal{F}_{a,B}$, meaning that for every $b\in B$ there is $c\in B_1$ such that $\models\theta_{i_{a,B}}\left(a;c\right)$ and $\theta_{i_{a,B}}\left(x;c\right)\vdash\varphi\left(x;b\right)$. That is, $\tp_{\theta_{i_{a,B}}}^+\left(a/B_1\right)\vdash \tp_\varphi^+\left(a/B\right)$.

Finally, let $\theta\left(x;y_1,\ldots,y_\ell\right)$ be a formula that can code for any $\theta_i\left(x;y_1,\ldots,y_{\ell_i}\right)$ when parameters range over a set with at least two elements. For all $a\in\mathbb{M}^x$ and finite $B\subseteq\mathbb{M}^y$ with $\left|B\right|\geq2$, for which $\models\varphi\left(a;b\right)$ for all $b\in B$, there is $B_0\subseteq B$ with $2\leq\left|B_0\right|\leq \ell N+2$ (consisting of the coordinates of $B_1$ from the previous paragraph, and two points for coding) such that $\tp_\theta^+\left(a/B_0\right)\vdash \tp_\varphi^+\left(a/B\right)$, as desired.

\noindent(3) implies (1). This follows almost verbatim from the proof of (2) implies (1) in \cite[Theorem 21]{chernikov2015externally}. Let $I+d+J$ be an indiscernible sequence in $\mathbb M^y$, with $I$ and $J$ infinite, and $I+J$ indiscernible over $a\in\mathbb M^x$, and suppose $\models\varphi\left(a,b\right)$ for $b\in I+J$. Let $I_1\subset I$ with $\left|I_1\right|=N+1$. Then there is some $I_0\subseteq I_1$ such that $\left|I_0\right|\leq N$ and $\tp_\theta^+\left(a/I_0\right)\vdash \tp_\varphi^+\left(a/I_1\right)$. Let $b\in I_1\setminus I_0$. By indiscernibility of $I+d+J$, there is some $\sigma\in\text{Aut}\left(\mathbb M\right)$ such that $\sigma\left(I_1\right)\subset I+d+J$ and $\sigma\left(b\right)=d$. We have $\sigma\left(I_0\right)\subseteq I+J$, so by $a$-indiscernibility of $I+J$, $\models\theta\left(a,\sigma\left(c\right)\right)$ for every $c\in I_0^k$ for which $\models\theta\left(a,c\right)$, and hence $a\models\sigma\left(\tp_\varphi^+\left(a/I_1\right)\right)$. And $\varphi\left(x,b\right)\in \tp_\varphi^+\left(a/I_1\right)$, so $\varphi\left(x,d\right)\in \sigma\left(\tp_\varphi^+\left(a/I_1\right)\right)$, and hence $\models \varphi\left(a,d\right)$.
\end{proof}
\begin{problem}
Can the assumption that $T$ is NIP be omitted? (Note that the proof of (3) implies (1) does not use it.)
\end{problem}
Proposition \ref{prop: semieq iff fin breadth} immediately implies an analog of Fact \ref{fact: equation iff isolates positive} for semi-equations, telling us that $\varphi\left(x;y\right)$ is a one-sided strong honest definition for itself:
\begin{cor}\label{stronger honest def}
	A formula $\varphi\left(x,y\right)$  (equivalently, $\varphi^*(y,x)$) is a semi-equation
if and only if there is some $k\in\mathbb{N}$ such that: for every finite $B\subset\mathbb{ M}^{y}$
and $a\in\mathbb{ M}^{x}$ there is some $B_0 \subseteq B$ with $|B_0| \leq k$ such that  $ \tp_{\varphi}^+(a/B_0) \vdash \tp_{\varphi}^{+}\left(a/B\right)$.
\end{cor}
\section{Non weakly semi-equational valued fields}\label{sec: val fields}

In this section we demonstrate that many valued fields are not weakly semi-equational. By an \emph{$\ac$-valued field field} we mean a three-sorted structure $\left(K,k,\Gamma, \nu, \ac \right)$ in the Denef-Pas language, where $K$ is a field, $\nu: K \to \Gamma$ is a valuation, with (ordered) value group $\Gamma$ and residue field $k$, and $\ac: K \to k$ the angular component map. As usual, $\mathcal{O} = \mathcal{O}_{\nu}$ denotes the valuation ring of $\nu$, and for $x \in \mathcal{O}$, $\bar{x}$ denotes the residue of $x$ in $k$. The following is the main theorem of the section:
\begin{theorem}\label{thm: non weak semieq val fields}
	Let $K$ be an ac-valued field for which the residue
field $k$ contains a non-constant totally indiscernible sequence (for
instance, if $k$ is infinite and stable), and which eliminates quantifiers
of the main field sort (for example, a Henselian
ac-valued field of equicharacteristic $0$ with an algebraically closed
residue field). Then $K$ is not weakly semi-equational.
\end{theorem}
\noindent Before presenting its proof, we need to develop some auxiliary results. First we provide a general sufficient criterion for when a formula is not a Boolean combination of weak semi-equations in Section \ref{sec: MS simple}. Then we discuss valuational independence in Section \ref{sec: val fields}. In Section \ref{sec: proof of nwsemieq val fields} we describe a particular configuration of elements in a valued field indented to satisfy this sufficient criterion with respect to the formula $\psi(x_1,x_2; y_1,y_2) := \nu\left(x_{1}-y_{1}\right)<\nu\left(x_{2}-y_{2}\right)$, and reduce demonstrating that it has all of the required properties to Claims  \ref{cla: val fields 1} and Claim \ref{cla: val fields 2} which express a certain amount of indiscernibility of our configuration. We also explain  how both claims can be proved by induction on the complexity of the formula and reduce to several essential cases that have to be considered; and show in Claim \ref{cla: val fields 4} valuational independence of some elements of our configuration which will be helpful in the proof of the claims.
We then prove Claim \ref{cla: val fields 1} in Section \ref{sec: claim 1} and Claim \ref{cla: val fields 2} in Section \ref{sec: claim 2}, concluding the proof of Theorem \ref{thm: non weak semieq val fields}. Finally, in Section \ref{sec: apps of main thm} we discuss some further applications of Theorem \ref{thm: non weak semieq val fields} and examples.

\subsection{Boolean combinations of weak semi-equations}\label{sec: MS simple}
We provide a sufficient criterion for when a formula is not a Boolean combination of weak semi-equations (analogous to a criterion for equations from \cite{muller2017nonequational}).

\begin{lemma}\label{lem: weak semieq conj}
	If $\varphi\left(x,y\right)$ and $\psi\left(x,y\right)$
are weak semi-equations, then there are no $b \in \mathbb{M}^{y}$ and array $\left(a_{i,j}\right)_{i,j\in\mathbb{Z}}$ with $a_{i,j} \in \mathbb{M}^x$ such that:
\begin{itemize}
	\item every row (i.e.~$(a_{i,j} : j \in \mathbb{Z})$ for a fixed $i \in \mathbb{Z}$) and every column (i.e.~$(a_{i,j} : i \in \mathbb{Z})$ for a fixed $j \in \mathbb{Z}$) is indiscernible (over $\emptyset$);

\item rows and columns without their $0$-indexed elements
(i.e.~$\left(a_{i,j}\right)_{j\neq0}$ for fixed $i$, and $\left(a_{i,j}\right)_{i\neq0}$
for fixed $j$) are $b$-indiscernible;

\item $\models\varphi\left(a_{i,j},b\right)\land\neg\psi\left(a_{i,j},b\right)\iff i=0\lor j\neq0$.
\end{itemize}
\end{lemma}

\begin{proof}
	Assume there exist an array $\left( a_{i,j} : i,j \in \mathbb{Z} \right)$ and $b$ with these properties. For any fixed $i\neq0$, we have $\models\varphi\left(a_{i,j},b\right)$
for all $j \neq 0$, $\left(a_{i,j}\right)_{j \in \mathbb{Z}}$ is indiscernible and $\left(a_{i,j}\right)_{j\neq0}$ is $b$-indiscernible,
so, by weak semi-equationality of $\varphi$, $\models\varphi\left(a_{i,0},b\right)$.
But $\not\models\varphi\left(a_{i,0},b\right)\land\neg\psi\left(a_{i,0},b\right)$,
so $\models\psi\left(a_{i,0},b\right)$. Now the sequence $\left(a_{i,0}\right)_{i \in \mathbb{Z}}$ is indiscernible, $\left(a_{i,0}\right)_{i\neq0}$
is $b$-indiscernible and $\models \psi(a_{i,0},b)$ for all $i \neq 0$, so, by weak semi-equationality of $\psi$, $\models\psi\left(a_{0,0},b\right)$ --- 
contradicting $\models\varphi\left(a_{0,0},b\right)\land\neg\psi\left(a_{0,0},b\right)$.
\end{proof}

\begin{lemma}\label{lem: Bool comb of weak semieq}
	If $\varphi\left(x,y\right)$ is a Boolean combination
of weak semi-equations, then there are no  $b \in \mathbb{M}^y$ and array $\left(a_{i,j}\right)_{i,j\in\mathbb{Z}}$ with $a_{i,j} \in \mathbb{M}^x$
such that:
\begin{itemize}
	\item rows and columns of $\left(a_{i,j}\right)_{i,j\in\mathbb{Z}}$
are indiscernible;
\item rows and columns without their $0$-indexed elements
(i.e. $\left(a_{i,j}\right)_{j\neq0}$ for fixed $i$, and $\left(a_{i,j}\right)_{i\neq0}$
for fixed $j$) are $b$-indiscernible;

\item  $\models\varphi\left(a_{i,j},b\right)\iff i=0\lor j\neq0$;

\item all $a_{i,j}$ with $i=0$ or $j\neq0$ have the same type over $b$.
\end{itemize}
\end{lemma}

\begin{proof}
	Any conjunction of finitely many weak semi-equations and negations
of weak semi-equations is of the form $\psi\left(x,y\right)\land\neg\theta\left(x,y\right)$
for some weak semi-equations $\psi\left(x,y\right)$ and $\theta\left(x,y\right)$,
because weak semi-equations are closed under conjunction and under disjunction (Proposition \ref{prop: Bool combs}(3)),
so negations of weak semi-equations are also closed under conjunction.
Thus any Boolean combination of weak semi-equations is equivalent, via
its disjunctive normal form, to $\bigvee_{k\in I} \left( \psi_{k}\left(x,y\right)\land\neg\theta_{k}\left(x,y\right) \right)$
for some finite index set $I$ and weak semi-equations $\psi_{k}\left(x,y\right)$
and $\theta_{k}\left(x,y\right)$ for $k\in I$. Given $b$ and $\left(a_{i,j}\right)_{i.j\in\mathbb{Z}}$
as above, since $i=0\lor j\neq0\iff\models\varphi\left(a_{i,j},b\right)\iff\models\bigvee_{k\in I} \left( \psi_{k}\left(a_{i,j},b\right)\land\neg\theta_{k}\left(a_{i,j},b\right) \right)$,
and all $a_{i,j}$ with $i=0$ or $j\neq0$ have the same type over $b$,
there is some $k$ such that $\models\psi_{k}\left(a_{i,j},b\right)\land\neg\theta_{k}\left(a_{i,j},b\right)\iff i=0\lor j\neq0$,
contradicting Lemma \ref{lem: weak semieq conj}.
\end{proof}

\subsection{Valuational independence}\label{sec: val fields}

\begin{defn}
	Let $K$ be a field with valuation $\nu$. 
	\begin{enumerate}
		\item We say that
$a_{1}, \ldots, a_{n}\in K$ are \emph{valuationally independent} if, for every
polynomial $f\left(x_{1}, \ldots, x_{n}\right)=\sum_{i}c_{i}x_{1}^{\alpha_{1,i}} \ldots x_{n}^{\alpha_{n,i}}$
(where $i$ runs over some finite index set, $c_{i},\alpha_{1,i}, \ldots, \alpha_{n,i}\in\mathbb{Z}$,
and $\left(\alpha_{1,i}, \ldots, \alpha_{n,i}\right)\neq\left(\alpha_{1,j}, \ldots, \alpha_{n,j}\right)$
for $i\neq j$) we have 
$$\nu\left(f\left(a_{1}, \ldots, a_{n}\right)\right)=\min_{i}\nu\left(c_{i}a_{1}^{\alpha_{1,i}} \ldots a_{n}^{\alpha_{n,i}}\right).$$
That is, if the valuation of every polynomial applied to $a_{1}, \ldots, a_{n}$
is the minimum of the valuations of its monomials (including their
coefficients).

\item  An infinite set is \emph{valuationally independent} 
if every finite subset is.
	\end{enumerate}
\end{defn}

\begin{expl}
(1) A set of elements with valuation $0$ is valuationally
independent if and only if their residues are algebraically independent.

\noindent(2) In a valued field of pure characteristic, every set of elements whose valuations
are $\mathbb{Z}$-linearly independent is valuationally independent.
In mixed characteristic $\left(0,p\right)$, every set of elements
whose valuations, together with $\nu\left(p\right)$, are $\mathbb{Z}$-linearly
independent, is valuationally independent. In an ac-valued field,
this is the only way for a set of elements with angular component
$1$ to be valuationally independent.
\end{expl}

\subsection{Reducing the proof of Theorem \ref{thm: non weak semieq val fields} to two claims}\label{sec: proof of nwsemieq val fields}
We will show that the partitioned formula $\psi(x_1,x_2; y_1,y_2) := \nu\left(x_{1}-y_{1}\right)<\nu\left(x_{2}-y_{2}\right)$
is not a Boolean combination of weak semi-equations.

Without loss of generality we may assume that  $K$ is a monster model. By Lemma \ref{lem: Bool comb of weak semieq}, it suffices to find $b, b'$
and $\left(a_{i}\right)_{i\in\mathbb{Z}},\left(a_{j}'\right)_{j\in\mathbb{Z}}$ in $K$ 
such that the sequences $\left(a_{i}\right)_{i\in\mathbb{Z}}$ and $\left(a_{j}'\right)_{j\in\mathbb{Z}}$
are mutually indiscernible (so that rows and columns of the array $\left(a_{i}a_{j}'\right)_{i,j\in\mathbb{Z}}$
are indiscernible), $\left(a_{i}\right)_{i\neq0}$ is indiscernible
over $bb'\left(a_{j}'\right)_{j\in\mathbb{Z}}$, $\left(a_{j}'\right)_{j\neq0}$
is indiscernible over $b b'\left(a_{i}\right)_{i\in\mathbb{Z}}$
(so that the rows and the columns of the array $\left(a_{i}a_{j}'\right)_{i,j\in\mathbb{Z}}$
with their $0$-indexed elements removed are indiscernible over $bb'$),
$\models\nu\left(a_{i}-b\right)<\nu\left(a_{j}'-b'\right)\iff i \neq 0\lor j = 0$,
and all pairs $\left(a_{i},a_{j}'\right)$ with $i \neq 0$ or $j = 0$ have
the same type over $bb'$.

To find these elements, first let $0 < \gamma_{0}<\gamma_{1}<\gamma_{2}<\gamma_{3}<\gamma_{4}<\gamma_{5}<\gamma_{6}\in\Gamma$
be an increasing \emph{indiscernible} sequence of positive elements of the
value group (exists by Ramsey and saturation). 

\begin{claim}\label{cla: val fields 3}
The elements	$\gamma_{0},\gamma_{1},\gamma_{2},\gamma_{3},\gamma_{4},\gamma_{5}$
are $\mathbb{Z}$-linearly independent (in $\Gamma$ viewed as a $\mathbb{Z}$-module). If $K$ has mixed characteristic
$\left(0,p\right)$, then $\nu\left(p\right),\gamma_{0},\gamma_{1},\gamma_{2},\gamma_{3},\gamma_{4},\gamma_{5}$
are $\mathbb{Z}$-linearly independent.
\end{claim}
\begin{proof}
	If $n_{0}\gamma_{0}+ \ldots + n_{5}\gamma_{5}=0$ with
$n_{0}, \ldots, n_{5}\in\mathbb{Z}$ not all $0$, let $i\leq5$ be maximal
such that $n_{i}\neq0$. Now $n_{0}\gamma_{0}+ \ldots + n_{i}\gamma_{i}=0$,
and $n_{i}\neq0$. By indiscernibility of the sequence $\left(\gamma_1, \ldots, \gamma_6 \right)$, $n_{0}\gamma_{0}+ \ldots + n_{i-1}\gamma_{i-1}+n_{i}\gamma_{6}=0$,
but then $n_{i}\left(\gamma_{i}-\gamma_{6}\right)=0$, contradicting
that $n_{i}\neq0$, $\gamma_{i}\neq\gamma_{6}$, and $\Gamma$ is
ordered and thus torsion-free. In mixed characteristic, the same argument
can be repeated starting from $n_{0}\gamma_{0}+ \ldots +n_{5}\gamma_{5}=m\nu\left(p\right)$
with $n_{0}, \ldots, n_{5},m\in\mathbb{Z}$.
\end{proof}

Next let $a_{\infty},a_{\infty}'\in K$ be such that $\nu\left(a_{\infty}\right)=\gamma_{0}$,
$\nu\left(a_{\infty}'\right)=\gamma_{1}$, and $\ac\left(a_{\infty}\right)=\ac\left(a_{\infty}'\right)=1$.
Let $\left(\tilde{a}_{i}\right)_{i\in\mathbb{Z}} + \left(\tilde{b} \right)$
and $\left(\tilde{a}_{j}'\right)_{j\in\mathbb{Z}} + \left( \tilde{b}' \right)$ 
be arbitrary mutually totally indiscernible sequences in the residue field  $k$. Such sequences exist by assumption on $k$ and saturation,  e.g.~splitting a totally indiscernible sequence into two disjoint subsequences.  We define $a_{i}:=a_{\infty}+\alpha\lift\left(\tilde{a}_{i}\right)$
and $a_{j}':=a_{\infty}'+\beta\lift\left(\tilde{a}_{j}'\right)$
for $i,j\in\mathbb{Z}$, for some $\alpha,\beta\in K$ with $\nu\left(\alpha\right)=\gamma_{2}$,
$\nu\left(\beta\right)=\gamma_{3}$, and $\ac\left(\alpha\right)=\ac\left(\beta\right)=1$. Here 
 $\lift\left(x\right)$ is some arbitrary element of $\mathcal{O}$
such that $\overline{\lift\left(x\right)}=x$. Let $b,b'$
be such that $\nu\left(a_{0}-b\right)=\gamma_{4}$, $\nu\left(a_{0}'-b'\right)=\gamma_{5}$, $\ac\left(a_{0}-b\right)=\tilde{b}-\tilde{a}_{0}$ and
$\ac\left(a_{0}'-b'\right)=\tilde{b}'-\tilde{a}_{0}'$. All of these elements are fixed for the rest of the section.

It is clear that $\models\nu\left(a_{i}-b\right)<\nu\left(a_{j}'-b'\right)\iff i \neq 0 \lor j = 0$,
because $\nu\left(a_{0}-b\right)=\gamma_{4}$, $\nu\left(a_{i}-b\right)=\gamma_{2}$
for $i\neq0$, $\nu\left(a_{0}'-b'\right)=\gamma_{5}$, and
$\nu\left(a_{j}'-b'\right)=\gamma_{3}$ for $j\neq0$.

We will prove the following two claims. Given a sequence $(x_i)_{i \in I}$ and $J \subseteq I$, we will write $x_{J}$ to denote the subsequence $\left(x_i : i \in J \right)$.
\begin{claim}\label{cla: val fields 1}
(1) Let $\varphi\left(x;z;w;b',a_{J}'\right)$ be a formula with parameters $b'$ and $a_{J}'$ for some $J\subseteq\mathbb{Z}$, tuples
of variables $x$ of sort $K$, $z$ of sort $k$, and $w$
of sort $\Gamma_{\infty}$. Let $I_{1},I_{2}$ be tuples
of distinct indices from $\mathbb{Z}$, with $\left|I_{1}\right|=\left|I_{2}\right|=\left|x\right|$. Let $\sigma\in \Aut\left(k\right)$
be such that $\sigma\left(\tilde{a}_{I_{1}}\right)=\tilde{a}_{I_{2}}$ (preserving the ordering of the tuples), $\sigma\left(\tilde{a}_{J}'\right)=\tilde{a}_{J}'$, and $\sigma\left(\tilde{b}'\right)=\tilde{b}'$.
Then  for any tuples $c \in k^z, d \in \Gamma_{\infty}^w$ we have
$\models\varphi\left(a_{I_{1}};c;d;b';a_{J}'\right)\iff\models\varphi\left(a_{I_{2}};\sigma\left(c\right);d;b';a_{J}'\right)$.

\noindent(2) Likewise, let $\varphi\left(y;z;w;b,a_{I}\right)$ be a formula with parameters $b$ and $a_I$ for some $I\subseteq\mathbb{Z}$, tuples of variables $y$ of sort $K$, $z$ of sort $k$, and $w$ of sort $\Gamma_\infty$. Let $J_1,J_2$ be tuples of distinct indices from $\mathbb{Z}$, with $\left|J_1\right|=\left|J_2\right|=\left|y\right|$, and let $\sigma\in\Aut\left(k\right)$ be such that $\sigma\left(\tilde{a}_{J_1}'\right)=\tilde{a}'_{J_2}$, $\sigma\left(\tilde{a}_I\right)=\tilde{a}_I$, and $\sigma\left(\tilde{b}\right)= \tilde{b}$. Then for any tuples $c \in k^z, d \in \Gamma_{\infty}^w$ we have 
$$\models\varphi\left(a_{J_{1}}';c;d;b;a_{I}\right)\iff\models\varphi\left(a_{J_{2}}';\sigma\left(c\right);d;b;a_{I}\right).$$
\end{claim}

\begin{claim}\label{cla: val fields 2}
	Let $\varphi\left(x;y;z;w;b;b'\right)$ be a formula with
parameters $b,b'$, where $x$ and $y$ are single variables of
sort $K$, and $z$ and $w$ are tuples of variables of sort $k$
and $\Gamma_{\infty}$, respectively. Let $\sigma_{i}\in \Aut \left(k\right)$ be 
such that $\sigma_{i}\left(\tilde{a}_{i}\right)=\tilde{b}$, $\sigma_{i}\left(\tilde{a}_{0}\right)=\tilde{a}_{0}$,
$\sigma_{i}\left(\tilde{a}_{0}'\right)=\tilde{a}_{0}'$, and
$\sigma_{i}\left(\tilde{b}'\right)=\tilde{b}'$. Let $\sigma_{j}'\in \Aut \left(k\right)$ be 
such that $\sigma_{j}'\left(\tilde{b}'\right)=\tilde{a}'_j$,
$\sigma_{j}'\left(\tilde{a}_{0}'\right)=\tilde{a}_{0}'$,
$\sigma_{j}'\left(\tilde{a}_{i}\right)=\tilde{a}_{i}$, and $\sigma_{j}'\left(\tilde{b}\right)=\tilde{b}$.
Let $\pi\in \Aut \left(\Gamma_{\infty}\right)$ be such that $\pi\left(\gamma_{2}\right)=\gamma_{4}$,
$\pi\left(\gamma_{0}\right)=\gamma_{0}$, $\pi\left(\gamma_{1}\right)=\gamma_{1}$,
and $\pi\left(\gamma_{5}\right)=\gamma_{5}$, and let $\tau\in \Aut\left(\Gamma_{\infty}\right)$
be such that $\tau\left(\gamma_{5}\right)=\gamma_{3}$, $\tau\left(\gamma_{0}\right)=\gamma_{0}$,
$\tau\left(\gamma_{1}\right)=\gamma_{1}$, and $\tau\left(\gamma_{2}\right)=\gamma_{2}$.
Then, for $i,j\neq0$, $c\in k^{z}$, and $d\in\Gamma_{\infty}^{w}$,
$\models\varphi\left(a_{0};a_{0}';\sigma_{i}\left(c\right);\pi\left(d\right);b;b'\right)\iff\models\varphi\left(a_{i};a_{0}';c;d;b;b'\right)\iff\models\varphi\left(a_{i};a_{j}';\sigma_{j}'\left(c\right);\tau\left(d\right);b;b'\right)$.
\end{claim}

Assuming these two claims, from the $\left|z\right|=\left|w\right|=0$ case of Claim \ref{cla: val fields 1}, we get that $\left(a_{i}\right)_{i\in\mathbb{Z}}$ is totally indiscernible over $b' \left(a_{j}'\right)_{j\in\mathbb{Z}}$, and $\left(a_{j}'\right)_{j\in\mathbb{Z}}$ is totally indiscernible over $b \left(a_{i}\right)_{i\in\mathbb{Z}}$. In particular $\left(a_{i}\right)_{i\in\mathbb{Z}}$ and $\left(a_{j}'\right)_{j\in\mathbb{Z}}$
are mutually totally indiscernible.

In describing $\left(a_{i}\right)_{i\in\mathbb{Z}},\left(a_{j}'\right)_{j\in\mathbb{Z}},b,b'$, we have made exactly the same assumptions about $a_{0}$ as about
$b$, and the same assumptions about $a_{0}'$ as about $b'$,
in the sense that if we replace $a_{0}$ with $b$ or replace $a_{0}'$
with $b'$, the resulting elements $\left(a_{i}\right)_{i\in\mathbb{Z}},\left(a_{j}'\right)_{j\in\mathbb{Z}},b,b'$
could have come from the same construction. Thus, as Claim \ref{cla: val fields 1} implies
that $\left(a_{i}\right)_{i\neq0}$ is totally indiscernible over $a_0 b' \left(a_{j}'\right)_{j\in\mathbb{Z}}$, and $\left(a_{j}'\right)_{j\neq0}$ is totally indiscernible over $a_{0}'b\left(a_{i}\right)_{i\in\mathbb{Z}}$, it must also be the case that
$\left(a_{i}\right)_{i\neq0}$ is totally indiscernible over $b b' \left(a_{j}'\right)_{j\in\mathbb{Z}}$, and $\left(a_{j}'\right)_{j\neq0}$ is totally indiscernible over $b'b\left(a_{i}\right)_{i\in\mathbb{Z}}$.

From the $\left|z\right|=\left|w\right|=0$
case of Claim \ref{cla: val fields 2}, we get that 
$$\text{tp}\left(a_{i},a_{j}'/b,b'\right)=\text{tp}\left(a_{i},a_{0}'/b,b'\right)=\text{tp}\left(a_{0},a_{0}'/b,b'\right)$$
for $i,j\neq0$, hence all $\left(a_{i},a_{j}'\right)$ with $i \neq 0$
or $j = 0$ have the same type over $bb'$.

Thus these two claims establish the conditions needed for Lemma \ref{lem: Bool comb of weak semieq} to imply that $\nu\left(x_{1}-y_{1}\right)<\nu\left(x_{2}-y_{2}\right)$
is not a Boolean combination of weak semi-equations.

Both claims will be proved by induction on the parse tree of the formula $\varphi$ (without parameters).
There are five cases that must be considered:

\noindent \textbf{Case 1.} The formula $\varphi$ is of the form $t_{1}\leq t_{2}$, where $t_{1},t_{2}$
are terms of sort $\Gamma_{\infty}$. Such terms 
are $\mathbb{N}$-linear combinations of variables of sort $\Gamma_{\infty}$
and valuations of polynomials in variables of sort $K$; i.e.~of the form $\boldsymbol{n}\cdot x+\boldsymbol{m}\cdot\nu\left(f\left(y\right)\right)$,
where $x = \left(x_1, \ldots, x_{\ell_1} \right)$ is a tuple of variables of sort $\Gamma_{\infty}$, $y$
is a tuple of variables of sort $K$, $f$ is a tuple of polynomials $\left( f_1(y), \ldots, f_{\ell_{2}}(y) \right)$,
$\boldsymbol{n} = \left(n_1, \ldots, n_{\ell_1} \right)\in\mathbb{N}^{\left|x\right|}$, $\boldsymbol{m} = \left(m_1, \ldots, m_{\ell_2} \right)\in\mathbb{N}^{\left|f\right|}$, $\nu \left(f(y) \right)$ is an abbreviation for the tuple $\left(\nu \left( f_1(y) \right), \ldots, \nu \left( f_{\ell}(y) \right) \right)$, and ``$\cdot$'' is the dot product.

\noindent \textbf{Case 2.} $\varphi$ is of the form $t_{1}=_{k}t_{2}$, where $t_{1},t_{2}$
are terms of sort $k$. Terms of sort $k$ are polynomials applied
to variables of sort $k$ and angular components of terms of sort
$K$; i.e.~of the form $f\left(x,\ac\left(g\left(y\right)\right)\right)$,
where $f$ is a polynomial, $g = \left(g_1, \ldots, g_{\ell} \right)$ is a tuple of polynomials, $x$ is
a tuple of variables of sort $k$, $y$ is a tuple of variables
of sort $K$, and $\ac\left(g(y) \right)$ is an abbreviation for the tuple $\left(\ac(g_1(y)), \ldots, \ac(g_{\ell}(y)) \right)$. Since $t_{1}=_{k}t_{2}$ if and only if $t_{1}-t_{2}=_{k}0$, every
formula of this form is equivalent to a formula of the form $f\left(x,\ac\left(g\left(y\right)\right)\right)=_{k}0$.

\noindent \textbf{Case 3.} $\varphi$ is a Boolean combination of formulas for which
the claim holds.

\noindent \textbf{Case 4.} $\varphi$ is of the form $\exists u\,\psi$, with $u$
a variable of sort $k$, and the claim holds for $\psi$.

\noindent \textbf{Case 5.} $\varphi$ is of the form $\exists u\,\psi$, with $u$ 
a variable of sort $\Gamma_{\infty}$, and the claim holds for $\psi$.

\noindent There are four more cases for how $\varphi$ could be constructed, but they follow from the
previous five cases: $\varphi$ is of the form $t_{1}=_{\Gamma}t_{2}$, where $t_{1},t_{2}$ are terms of
sort $\Gamma_{\infty}$ --- this is equivalent to $t_{1}\leq t_{2}\land t_{2}\leq t_{1}$,
and is thus redundant with Cases 1 and 3;  $\varphi$ is of the form $t_{1}=_{K}t_{2}$, where $t_{1},t_{2}$ are terms of sort
$K$ --- this is equivalent to $\nu\left(t_{1}-t_{2}\right)=\nu\left(0\right)$,
and is thus redundant with Cases 1 and 3;  $\varphi$ is of the form $\forall u\,\psi$, where $u$ is a variable of sort $k$
or $\Gamma_{\infty}$ --- this is redundant with Cases 3, 4, and 5;  $\varphi$ is of the form $\exists u\,\psi$, or $\forall u\,\psi$, where $u$ is a
variable of sort $K$ --- this case can be neglected by quantifier elimination,
since we can always pick a formula equivalent to $\varphi$ which
has no quantifiers of sort $K$.

The following auxiliary result will be used in the proof of the claims.

\begin{claim}\label{cla: val fields 4}
The elements $a_{\infty},a_{\infty}',\left(\alpha \lift \left( \tilde{a}_{i} \right)\right)_{i\in \mathbb{Z}}, \left(\beta \lift \left( \tilde{a}_{j}' \right)\right)_{j\in \mathbb{Z}}, b-a_0, b'-a_{0}'$ are valuationally independent.
\end{claim}
\begin{proof}
	 Define a valuation $\nu^{*}:\mathbb{Z}\left[u,v,x,y,z,w\right]\rightarrow\Gamma_{\infty}$
(with $\left|u\right|=\left|v\right|=\left|z\right|=\left|w\right|=1$, $\left|x\right|$,$\left|y\right|$
arbitrary), by, for monomials (which in case of mixed characteristic
is taken to include its coefficient),
\begin{gather*}
	\nu^{*}\left(n \cdot u^{r_{\infty}}x_{1}^{r_{1}} \ldots x_{\left|x\right|}^{r_{\left|x\right|}}v^{s_{\infty}}y_{1}^{s_{1}} \ldots y_{\left|y\right|}^{s_{\left|y\right|}}z^{t_1}w^{t_2}\right) \\
	:=\nu\left(n\right)+r_{\infty}\gamma_{0}+s_{\infty}\gamma_{1}+\left(r_{1}+ \ldots +r_{\left|x\right|}\right)\gamma_{2}+\left(s_{1}+ \ldots +s_{\left|y\right|}\right)\gamma_{3} + t_1\gamma_{4} + t_2\gamma_{5},
\end{gather*}
and the valuation of a polynomial is the minimum of the valuations
of its monomials. 
That way, for any $I,J\subseteq\mathbb{Z}$ with
$\left|I\right|=\left|x\right|$ and $\left|J\right|=\left|y\right|$ we have:
\begin{gather*}
	\nu^{*}\left(f\left(u,v,x,y,z,w\right)\right)=\nu\left(f\left(a_{\infty},a_{\infty}',\alpha\cdot\lift\left(\tilde{a}_{I}\right),\beta\cdot\lift\left(\tilde{a}_{J}'\right),b-a_0,b'-a_{0}'\right)\right)
\end{gather*}
when $f$ is a monomial (where $\alpha \cdot \lift(\tilde{a}_I) := \left(\alpha \lift(\tilde{a}_i)\right)_{i\in I}$), and we need to prove that this holds for
all polynomials $f$. Given a polynomial $f\left(u,v,x,y,z,w\right)$,
$$\nu^{*}\left(f\right)=\nu\left(n\right)+m_{0}\gamma_{0}+m_{1}\gamma_{1}+m_{2}\gamma_{2}+m_{3}\gamma_{3}+m_4\gamma_4+m_5\gamma_5$$
for some $n,m_{0},m_{1},m_{2},m_{3},m_{4},m_{5}\in\mathbb{N}$ (with $\nu\left(n\right),m_{0},m_{1},m_{2},m_{3},m_{4},m_{5}$
unique by Claim \ref{cla: val fields 3}). Let $\tilde{f}\left(u,v,x,y,z,w\right)$ be the sum of monomials in $f$ of the same valuation as $f$, so that every monomial appearing in $\tilde{f}\left(u,v,x,y,z,w\right)$ has degree
$m_{0}$ in $u$, degree $m_{1}$ in $v$, total degree $m_{2}$ in
$x$, total degree $m_{3}$ in $y$, degree $m_{4}$ in $z$, degree $m_{5}$ in $w$, and has leading coefficient with
valuation $\nu\left(n\right)$, and $\nu^{*}\left(f-\tilde{f}\right)>\nu^{*}\left(f\right)$.
Thus 
$$\frac{\tilde{f}\left(u,v,x,y,z,w\right)}{n \cdot u^{m_{0}}v^{m_{1}}z^{m_4}w^{m_5}}$$
 is a non-zero polynomial in $x,y$, all coefficients having valuation
$0$, so it reduces under the residue map to a nonzero polynomial
in $x,y$. Since the set of elements in the tuples $\tilde{a}_{I},\tilde{a}_{J}'$ is algebraically
independent (they come from an infinite indiscernible sequence), it follows that 
$$\frac{\tilde{f}\left(u,v,\tilde{a}_{I},\tilde{a}_{J}',z,w\right)}{n \cdot u^{m_{0}}v^{m_{1}}z^{m_4}w^{m_5}}\neq0,$$
 and thus a lift of it, 
 $$\frac{\tilde{f}\left(u,v,\lift\left(\tilde{a}_{I}\right),\lift\left(\tilde{a}_{J}'\right),z,w\right)}{n \cdot u^{m_{0}}v^{m_{1}}z^{m_4}w^{m_5}},$$
has valuation $0$. Thus 
$$\nu\left(\tilde{f}\left(a_{\infty},a_{\infty}',\lift\left(\tilde{a}_{I}\right),\lift\left(\tilde{a}_{J}'\right),b-a_0,b'-a_{0}'\right)\right)=\nu\left(n\right)+m_{0}\gamma_{0}+m_{1}\gamma_{1}+m_4\gamma_4+m_5\gamma_5$$
and, by homogeneity of $\tilde{f}$, 
\begin{gather*}\nu\left(\tilde{f}\left(a_{\infty},a_{\infty}',\alpha\cdot\lift\left(\tilde{a}_{I}\right),\beta\cdot\lift\left(\tilde{a}_{J}'\right),b-a_0,b'-a_{0}'\right)\right)
\\=\nu\left(n\right)+m_{0}\gamma_{0}+m_{1}\gamma_{1}+m_{2}\gamma_{2}+m_{3}\gamma_{3}+m_4\gamma_4+m_5\gamma_5=\nu^{*}\left(f\right).\end{gather*}
We have 
$$\nu\left(\left(f-\tilde{f}\right)\left(a_{\infty},a_{\infty}',\alpha\cdot\lift\left(\tilde{a}_{I}\right),\beta\cdot\lift\left(\tilde{a}_{J}'\right),b-a_0,b'-a_{0}'\right)\right)\geq\nu^{*}\left(f-\tilde{f}\right)>\nu^{*}\left(f\right)$$
(the first inequality holds by the ultrametric property, combined
with the fact that it holds for monomials), so it follows that 
$$\nu\left(f\left(a_{\infty},a_{\infty}',\alpha\cdot\lift\left(\tilde{a}_{I}\right),\beta\cdot\lift\left(\tilde{a}_{J}'\right), b-a_0,b'-a_{0}' \right)\right)=\nu^{*}\left(f\right).\qedhere$$ 
\end{proof}
We are ready to prove the two claims.
\subsection{Proof of Claim \ref{cla: val fields 1}}\label{sec: claim 1}
Let $\varphi\left(x;z;w;b';a_{J}'\right)$ with $x = \left(x_1, \ldots, x_{|x|} \right)$ and $I_{1},I_{2}, \sigma, c, d$ be as in the statement of the claim, and we analyze the five cases described above. We will assume without loss of generality that $j_1 = 0$, where $J = (j_1, \ldots, j_{|J|})$ (since if $0$ appears somewhere else in $J$, $J$ may be re-ordered, and if $0$ does not appear in $J$, it may be added).
The proof for the part regarding a formula $\varphi\left(y;z;w;b;a_I\right)$ is identical, switching the roles of $\left(a_i\right)_{i\in\mathbb{Z}}$ and $\left(a_{j}'\right)_{j\in\mathbb{Z}}$, replacing $b'$ with $b'$, and replacing $\gamma_5$ with $\gamma_4$.
	
\noindent \textbf{Case 1.} $\varphi\left(x;z;w;b';a_{J}'\right)$ is of the form $\boldsymbol{n}_{1}\cdot w+\boldsymbol{m}_{1}\cdot\nu\left(g\left(x,b',a_{J}'\right)\right)\leq\boldsymbol{n}_{2}\cdot w+\boldsymbol{m}_{2}\cdot\nu\left(h\left(x,b',a_{J}'\right)\right)$.

It is enough to show that for any polynomial $f\left(x,q,y\right)$ (with $\left|x\right|=\left|I_1\right|,\left|y\right|=\left|J\right|,\left|q\right|=1$), we have 
$\nu\left(f\left(a_{I_{1}},b',a_{J}'\right)\right)=\nu\left(f\left(a_{I_{2}},b',a_{J}'\right)\right)$,
because then 
\begin{gather*}
	\boldsymbol{m}_{1}\cdot\nu\left(g\left(a_{I_{1}},b',a_{J}'\right)\right)=\boldsymbol{m}_{1}\cdot\nu\left(g\left(a_{I_{2}},b',a_{J}'\right)\right) \textrm{ and }\\
	\boldsymbol{m}_{2}\cdot\nu\left(h\left(a_{I_{1}},b',a_{J}'\right)\right)=\boldsymbol{m}_{2}\cdot\nu\left(h\left(a_{I_{2}},b',a_{J}'\right)\right), \textrm{ so}\\
	\models\boldsymbol{n}_{1}\cdot d+\boldsymbol{m}_{1}\cdot\nu\left(g\left(a_{I_{1}},b',a_{J}'\right)\right)\leq\boldsymbol{n}_{2}\cdot d+\boldsymbol{m}_{2}\cdot\nu\left(h\left(a_{I_{1}},b',a_{J}'\right)\right)\iff \\
	\models\boldsymbol{n}_{1}\cdot d+\boldsymbol{m}_{1}\cdot\nu\left(g\left(a_{I_{2}},b',a_{J}'\right)\right)\leq\boldsymbol{n}_{2}\cdot d+\boldsymbol{m}_{2}\cdot\nu\left(h\left(a_{I_{2}},b',a_{J}'\right)\right).
\end{gather*}
Given a polynomial $f(x,q,y)$, let 
$$f^{*}\left(u,v,x,y,q\right):=f\left(x_{1}+u, \ldots, x_{\left|x\right|}+u,q + y_1 + v,y_{1} + v, \ldots, y_{\left|y\right|} + v\right),$$
with $\left|u\right|=\left|v\right|=\left|q\right|=1$, $\left|x\right|=\left|I_{1}\right|$,
$\left|y\right|=\left|J\right|$, so that 
$$f^{*}\left(a_{\infty},a_{\infty}',\alpha\cdot\lift\left(\tilde{a}_{I_{i}}\right),\beta\cdot\lift\left(\tilde{a}_{J}'\right),b'-a_{0}'\right)=f\left(a_{I_{i}},b',a_{J}'\right)$$
for $i\in\left\{ 1,2\right\} $ (using that $a_i = a_\infty + \alpha\cdot\lift\left(\tilde{a}_i\right)$ and $a_j' = a_\infty' + \beta\cdot\lift\left(\tilde{a}_j'\right)$ and $j_1 = 0$).
\begin{gather*}
	\textrm{Since } \  \nu \Bigg( n \cdot a_{\infty}^{r_{\infty}}\left(\alpha\lift\left(\tilde{a}_{i_{1}}\right)\right)^{r_{1}} \ldots \left(\alpha\lift\left(\tilde{a}_{i_{\left|x\right|}}\right)\right)^{r_{\left|x\right|}}\left(a_{\infty}'\right)^{s_{\infty}}  \cdot \\
	\cdot \left(\beta\lift\left(\tilde{a}_{j_{1}}'\right)\right)^{s_{1}} \ldots \left(\beta\lift\left(\tilde{a}_{j_{\left|y\right|}}'\right)\right)^{s_{\left|y\right|}}\left(b'-a_{0}'\right)^t \Bigg) \\
	=\nu\left(n\right)+r_{\infty}\gamma_{0}+s_{\infty}\gamma_{1}+\left(r_{1}+ \ldots+ r_{\left|x\right|}\right)\gamma_{2}+\left(s_{1}+ \ldots + s_{\left|y\right|}\right)\gamma_{3}+t\gamma_5,
\end{gather*}
\noindent regardless of $i_{1}, \ldots, i_{\left|x\right|}$,
if we let 
$$n \cdot u^{r_{\infty}}v^{s_{\infty}}x_{1}^{r_{1}} \ldots x_{\left|x\right|}^{r_{\left|x\right|}}y_{1}^{s_{1}} \ldots y_{\left|y\right|}^{s_{\left|y\right|}}q^t$$
be a monomial in $f^{*}\left(u,v,x,y,q\right)$ minimizing 
$$\nu\left(n\right)+r_{\infty}\gamma_{0}+s_{\infty}\gamma_{1}+\left(r_{1}+ \ldots +r_{\left|x\right|}\right)\gamma_{2}+\left(s_{1}+ \ldots +s_{\left|y\right|}\right)\gamma_{3}+t\gamma_5,$$
then by Claim \ref{cla: val fields 4},
\begin{gather*}
	\nu\left(f\left(a_{I_{i}},b',a_{J}'\right)\right)=\nu\left(f^{*}\left(a_{\infty},a_{\infty}',\alpha\cdot\lift\left(\tilde{a}_{I_{i}}\right),\beta\cdot\lift\left(\tilde{a}_{J}'\right),b'-a_{0}'\right)\right)\\
	=\nu\left(n\right)+r_{\infty}\gamma_{0}+s_{\infty}\gamma_{1}+\left(r_{1}+ \ldots +r_{\left|x\right|}\right)\gamma_{2}+\left(s_{1}+ \ldots + s_{\left|y\right|}\right)\gamma_{3}+t\gamma_5
\end{gather*}
for $i\in\left\{ 1,2\right\} $.

\noindent \textbf{Case 2.} $\varphi\left(x;z;w;b';a_{J}'\right)$ is of the form $f\left(z,\ac\left(g\left(x,b,a_{J}'\right)\right)\right)=_{k}0$.

It is enough to show that $f\left(\sigma\left(c\right),\ac\left(g\left(a_{I_{2}},b',a_{J}'\right)\right)\right)=\sigma\left(f\left(c,\ac\left(g\left(a_{I_{1}},b',a_{J}'\right)\right)\right)\right)$,
for which it is in turn enough to show that $\ac\left(g\left(a_{I_{2}},b',a_{J}'\right)\right)=\sigma\left(\ac\left(g\left(a_{I_{1}},b',a_{J}'\right)\right)\right)$.
Since $a_{i}=a_{\infty}+\alpha\cdot\lift\left(\tilde{a}_{i}\right)$
and $a_{j}'=a_{\infty}'+\beta\cdot\lift\left(\tilde{a}_{j}'\right)$,
there is a polynomial $h\left(u,v,x,y,q\right)$ (with $\left|u\right|=\left|v\right|=\left|q\right|=1$,
$\left|x\right|=\left|I_{1}\right|$, $\left|y\right|=\left|J\right|$)
such that 
$$h\left(a_{\infty},a_{\infty}',\alpha\cdot\lift\left(\tilde{a}_{I_{i}}\right),\beta\cdot\lift\left(\tilde{a}_{J}'\right),b'-a_{0}'\right)=g\left(a_{I_{i}},b',a_{J}'\right)$$
for $i\in\left\{ 1,2\right\} $. As in the proof of Case 1, there
are $n,m_{0},m_{1},m_{2},m_{3},m_5\in\mathbb{N}$ such that 
\begin{gather*}
	\nu\left(h\left(a_{\infty},a_{\infty}',\alpha\cdot\lift\left(\tilde{a}_{I_{i}}\right),\beta\cdot\lift\left(\tilde{a}_{J}'\right),b'-a_{0}'\right)\right)\\
	=\nu\left(n\right)+m_{0}\gamma_{0}+m_{1}\gamma_{1}+m_{2}\gamma_{2}+m_{3}\gamma_{3}+m_5\gamma_5
\end{gather*}
for $i\in\left\{ 1,2\right\} $. Let $\tilde{h}\left(u,v,x,y,q\right)$
be the sum of monomials in $h$ with degree $m_{0}$ in $u$, degree
$m_{1}$ in $v$, total degree $m_{2}$ in $x$, total degree $m_{3}$
in $y$, degree $m_5$ in $q$, and whose coefficient has valuation $\nu\left(n\right)$.
That way 
\begin{gather*}
	\nu\left(\tilde{h}\left(a_{\infty},a_{\infty}',\alpha\cdot\lift\left(\tilde{a}_{I_{i}}\right),\beta\cdot\lift\left(\tilde{a}_{J}'\right),b'-a_{0}'\right)\right)\\
	=\nu\left(n\right)+m_{0}\gamma_{0}+m_{1}\gamma_{1}+m_{2}\gamma_{2}+m_{3}\gamma_{3}+m_5\gamma_5, \textrm{ and}\\
	\nu\left(\left(h-\tilde{h}\right)\left(a_{\infty},a_{\infty}',\alpha\cdot\lift\left(\tilde{a}_{I_{i}}\right),\beta\cdot\lift\left(\tilde{a}_{J}'\right),b'-a_{0}'\right)\right)\\
	>\nu\left(n\right)+m_{0}\gamma_{0}+m_{1}\gamma_{1}+m_{2}\gamma_{2}+m_{3}\gamma_{3}+m_5\gamma_5.
\end{gather*}
Then 
$h^{*}\left(x,y\right):=\frac{\tilde{h}\left(u,v,x,y,q\right)}{n \cdot u^{m_{0}}v^{m_{1}}q^{m_5}}$
is a non-zero polynomial in $x,y$, all coefficients having valuation
$0$, so it reduces under the residue map to a nonzero polynomial
in $x,y$. Since $\tilde{a}_{I_{i}},\tilde{a}_{J}'$ are algebraically
independent (by indiscernibility), it follows that $\overline{h^{*}}\left(\tilde{a}_{I_{i}},\tilde{a}_{J}'\right)\neq0$,
so $h^{*}\left(\lift\left(\tilde{a}_{I_{i}}\right),\lift\left(\tilde{a}_{J}'\right)\right)$
has valuation $0$, and hence its angular component is its residue,
$\overline{h^{*}}\left(\tilde{a}_{I_{i}},\tilde{a}_{J}'\right)$.
We have
\begin{gather*}
\overline{h^{*}}\left(\tilde{a}_{I_{2}},\tilde{a}_{J}'\right)=\overline{h^{*}}\left(\sigma\left(\tilde{a}_{I_{1}}\right),\sigma\left(\tilde{a}_{J}'\right)\right)=\sigma\left(\overline{h^{*}}\left(\tilde{a}_{I_{1}},\tilde{a}_{J}'\right)\right) \textrm{, thus }\\
	\ac\left(\tilde{h}\left(a_{\infty},a_{\infty}',\alpha\cdot\lift\left(\tilde{a}_{I_{2}}\right),\beta\cdot\lift\left(\tilde{a}_{J}'\right),b'-a_{0}'\right)\right)\\
	=\ac\left(n \cdot a_{\infty}^{m_{0}}\left(a_{\infty}'\right)^{m_{1}}\alpha^{m_{2}}\beta^{m_{3}}\left(b'-a_{0}'\right)^{m_5}\right)\ac\left(\frac{\tilde{h}\left(a_{\infty},a_{\infty}',\lift\left(\tilde{a}_{I_{2}}\right),\lift\left(\tilde{a}_{J}'\right),b'-a_{0}'\right)}{n \cdot a_{\infty}^{m_{0}}\left(a_{\infty}'\right)^{m_{1}}\left(b'-a_{0}'\right)^{m_5}}\right)\\
	=\ac\left(n\right)\left(\tilde{a}_{0}'-\tilde{b}'\right)^{m_5}\overline{h^{*}}\left(\tilde{a}_{I_{2}},\tilde{a}_{J}'\right)
	=\ac\left(n\right)\left(\tilde{a}_{0}'-\tilde{b}'\right)^{m_5}\sigma\left(\overline{h^{*}}\left(\tilde{a}_{I_{1}},\tilde{a}_{J}'\right)\right)\\
	=\sigma\left(\ac\left(n\right)\left(\tilde{a}_{0}'-\tilde{b}'\right)^{m_5}\overline{h^{*}}\left(\tilde{a}_{I_{1}},\tilde{a}_{J}'\right)\right)\\
	=\sigma \Bigg(\ac\left(n \cdot a_{\infty}^{m_{0}}\left(a_{\infty}'\right)^{m_{1}}\alpha^{m_{2}}\beta^{m_{3}}\left(b'-a_{0}'\right)^{m_5}\right) \cdot \\\cdot \ac\left(\frac{\tilde{h}\left(a_{\infty},a_{\infty}',\lift\left(\tilde{a}_{I_{1}}\right),\lift\left(\tilde{a}_{J}'\right),b'-a_{0}'\right)}{n \cdot a_{\infty}^{m_{0}}\left(a_{\infty}'\right)^{m_{1}}\left(b'-a_{0}'\right)^{m_5}}\right)\Bigg)\\
	=\sigma\left(\ac\left(\tilde{h}\left(a_{\infty},a_{\infty}',\alpha\cdot\lift\left(\tilde{a}_{I_{1}}\right),\beta\cdot\lift\left(\tilde{a}_{J}'\right),b'-a_{0}'\right)\right)\right).
	\end{gather*}
	\begin{gather*}
	\textrm{Since } \ \nu\left(\left(h-\tilde{h}\right)\left(a_{\infty},a_{\infty}',\alpha\cdot\lift\left(\tilde{a}_{I_{i}}\right),\beta\cdot\lift\left(\tilde{a}_{J}'\right),b'-a_{0}'\right)\right)\\
	>\nu\left(h\left(a_{\infty},a_{\infty}',\alpha\cdot\lift\left(\tilde{a}_{I_{i}}\right),\beta\cdot\lift\left(\tilde{a}_{J}'\right),b'-a_{0}'\right)\right),\\
\textrm{we have } \ \ac\left(g\left(a_{I_{i}},b',a_{J}'\right)\right)=\ac\left(h\left(a_{\infty},a_{\infty}',\alpha\cdot\lift\left(\tilde{a}_{I_{i}}\right),\beta\cdot\lift\left(\tilde{a}_{J}'\right),b'-a_{0}'\right)\right)	\\
=\ac\left(\tilde{h}\left(a_{\infty},a_{\infty}',\alpha\cdot\lift\left(\tilde{a}_{I_{i}}\right),\beta\cdot\lift\left(\tilde{a}_{J}'\right),b'-a_{0}'\right)\right) \textrm{, hence}\\
	\ac\left(g\left(a_{I_{2}},b',a_{J}'\right)\right)=\ac\left(\tilde{h}\left(a_{\infty},a_{\infty}',\alpha\cdot\lift\left(\tilde{a}_{I_{2}}\right),\beta\cdot\lift\left(\tilde{a}_{J}'\right),b-a_{0}'\right)\right)\\
	=\sigma\left(\ac\left(\tilde{h}\left(a_{\infty},a_{\infty}',\alpha\cdot\lift\left(\tilde{a}_{I_{1}}\right),\beta\cdot\lift\left(\tilde{a}_{J}'\right),b'-a_{0}'\right)\right)\right)=\sigma\left(\ac\left(g\left(a_{I_{1}},b',a_{J}'\right)\right)\right).
\end{gather*}

\noindent \textbf{Case 3.} Clear.

\noindent \textbf{Case 4.} $\varphi\left(x;z;w;b';a_{J}'\right)$ is of the form $\exists u\,\psi\left(x;z,u;w;b';a_{J}'\right)$,
where $u$ is a variable of sort $k$, and the claim holds for $\psi$.
If $\models\varphi\left(a_{I_{1}};c;d;b';a_{J}'\right)$, then
$\models\psi\left(a_{I_{1}};c,e;d;b';a_{J}'\right)$ for some
$e\in k$. Then we have $\models\psi\left(a_{I_{2}};\sigma\left(c\right),\sigma\left(e\right);d;b';a_{J}'\right)$,
so $\models\varphi\left(a_{I_{2}};\sigma\left(c\right);d;b';a_{J}'\right)$.

\noindent \textbf{Case 5.} $\varphi\left(x;z;w;b';a_{J}'\right)$ is of the form $\exists u\,\psi\left(x;z;w,u;b';a_{J}'\right)$,
where $u$ is a variable of sort $\Gamma_{\infty}$, and the claim
holds for $\psi$. 
If $\models\varphi\left(a_{I_{1}};c;d;b';a_{J}'\right)$,
then $\models\psi\left(a_{I_{1}};c;d,e;b';a_{J}'\right)$ for
some $e\in\Gamma_{\infty}$. Then we have $\models\psi\left(a_{I_{2}};\sigma\left(c\right);d,e;b';a_{J}'\right)$,
so $\models\varphi\left(a_{I_{2}};\sigma\left(c\right);d;b';a_{J}'\right)$. This concludes the proof of Claim \ref{cla: val fields 1}.

\subsection{Proof of Claim \ref{cla: val fields 2}}\label{sec: claim 2}

Let $\varphi\left(x;y;z;w;b;b'\right), \sigma_i, \sigma'_j, \pi, \tau$ be as in the claim.

\noindent \textbf{Case 1.} $\varphi\left(x;y;z;w;b;b'\right)$ is of the form $\boldsymbol{n}_{1}\cdot w+\boldsymbol{m}_{1}\cdot\nu\left(g\left(x,y,b,b'\right)\right)\leq\boldsymbol{n}_{2}\cdot w+\boldsymbol{m}_{2}\cdot\nu\left(h\left(x,y,b,b'\right)\right)$.
It is enough to show that for any polynomial $f\left(x,y,u,u'\right)$,
$$\pi^{-1}\left(\nu\left(f\left(a_{0},a_{0}',b,b'\right)\right)\right)=\nu\left(f\left(a_{i},a_{0}',b,b'\right)\right)=\tau^{-1}\left(\nu\left(f\left(a_{i},a_{j}',b,b'\right)\right)\right)$$
for $i,j\neq0$, because then 
\begin{gather*}
	\boldsymbol{n}_{1}\cdot\pi\left(d\right)+\boldsymbol{m}_{1}\cdot\nu\left(g\left(a_{0},a_{0}',b,b'\right)\right)=\pi\left(\boldsymbol{n}_{1}\cdot d+\boldsymbol{m}_{1}\cdot\nu\left(g\left(a_{i},a_{0}',b,b'\right)\right)\right) \textrm{ and}\\
	\boldsymbol{n}_{1}\cdot\tau\left(d\right)+\boldsymbol{m}_{1}\cdot\nu\left(g\left(a_{i},a_{j}',b,b'\right)\right)=\tau\left(\boldsymbol{n}_{1}\cdot d+\boldsymbol{m}_{1}\cdot\nu\left(g\left(a_{i},a_{0}',b,b'\right)\right)\right)
\end{gather*}

\noindent for $i,j\neq0$, and likewise for $\boldsymbol{n}_{2},\boldsymbol{m}_{2},h$,
and, as $\pi$ and $\tau$ preserve order, this implies 
\begin{gather*}
	\models\varphi\left(a_{0};a_{0}';\sigma_i\left(c\right);\pi\left(d\right);b;b'\right)\iff\models\varphi\left(a_{i};a_{0}';c;d;b;b'\right)\\\iff\models\varphi\left(a_{i};a_{j}';\sigma'_j\left(c\right);\tau\left(d\right);b;b'\right).
\end{gather*}
To show this, let $f^{*}\left(x,y,u,v\right):=f\left(x+u,y+v,u,v\right)$. 
By Claim
\ref{cla: val fields 3} and the choice of these elements, for $i,j\in\mathbb{Z}$, the valuations of $a_{i}-b$, $a_{j}'-b'$,
$b$, and $b'$ are $\mathbb{Z}$-linearly independent (together
with $\nu\left(p\right)$ if the characteristic is mixed), and hence
these are valuationally independent. Let $nx^{e_{1}}y^{e_{2}}u^{e_{3}}v^{e_{4}}$
be the monomial in $f^{*}\left(x,y,u,v\right)$ minimizing $\nu\left(n\right)+e_{1}\gamma_{2}+e_{2}\gamma_{5}+e_{3}\gamma_{0}+e_{4}\gamma_{1}$,
so that by valuational independence, 
$$\nu\left(f^{*}\left(a_{i}-b,a_{0}'-b',b,b'\right)\right)=\nu\left(n\right)+e_{1}\gamma_{2}+e_{2}\gamma_{5}+e_{3}\gamma_{0}+e_{4}\gamma_{1}.$$
This monomial is unique by linear independence (Claim \ref{cla: val fields 3}). Since $\pi$
and $\tau$ preserve order, this monomial also minimizes 
\begin{gather*}
	\pi\left(\nu\left(n\right)+e_{1}\gamma_{2}+e_{2}\gamma_{5}+e_{3}\gamma_{0}+e_{4}\gamma_{1}\right)=\nu\left(n\right)+e_{1}\gamma_{4}+e_{2}\gamma_{5}+e_{3}\gamma_{0}+e_{4}\gamma_{1}\\
	=\nu\left(f^{*}\left(a_{0}-b,a_{0}'-b',b,b'\right)\right)\textrm{, and}\\
	\tau\left(\nu\left(n\right)+e_{1}\gamma_{2}+e_{2}\gamma_{5}+e_{3}\gamma_{0}+e_{4}\gamma_{1}\right)=\nu\left(n\right)+e_{1}\gamma_{2}+e_{2}\gamma_{3}+e_{3}\gamma_{0}+e_{4}\gamma_{1}\\
	=\nu\left(f^{*}\left(a_{i}-b,a_{j}'-b',b,b'\right)\right) \textrm{ for } i,j\neq0.
\end{gather*}

\noindent \textbf{Case 2.} $\varphi\left(x;y;z;w;b;b'\right)$ is of the form $f\left(z,\ac\left(g\left(x,y,b,b'\right)\right)\right)=_{k}0$.

It is enough to show that 
\begin{gather*}
	\sigma_{i}\left(f\left(c,\ac\left(g\left(a_{i},a_{0}',b,b'\right)\right)\right)\right)=f\left(\sigma_{i}\left(c\right),\ac\left(g\left(a_{0},a_{0}',b,b'\right)\right)\right) \textrm{ and}\\
	\sigma_{j}'\left(f\left(c,\ac\left(g\left(a_{i},a_{0}',b,b'\right)\right)\right)\right)=f\left(\sigma_{j}'\left(c\right),\ac\left(g\left(a_{i},a_{j}',b,b'\right)\right)\right),
\end{gather*}
for which it is in turn enough to show that 
\begin{gather*}
	\sigma_{i}\left(\ac\left(g\left(a_{i},a_{0}',b,b'\right)\right)\right)=\ac\left(g\left(a_{0},a_{0}',b,b'\right)\right) \textrm{ and}\\
	\sigma_{j}'\left(\ac\left(g\left(a_{i},a_{0}',b,b'\right)\right)\right)=\ac\left(g\left(a_{i},a_{j}',b,b'\right)\right).
\end{gather*}

\noindent Let $h\left(x,y,u,v\right):=g\left(x+u,y+v,u,v\right)$. Let $nx^{m_{1}}y^{m_{2}}u^{m_{3}}v^{m_{4}}$
be the (unique, by Claim \ref{cla: val fields 3}) monomial in $h\left(x,y,u,v\right)$ minimizing
$\nu\left(n\right)+m_{1}\gamma_{2}+m_{2}\gamma_{5}+m_{3}\gamma_{0}+m_{4}\gamma_{1}$,
so that by valuational independence, $\nu\left(h\left(a_{i}-b,a_{0}'-b',b,b'\right)\right)=\nu\left(n\right)+m_{1}\gamma_{2}+m_{2}\gamma_{5}+m_{3}\gamma_{0}+m_{4}\gamma_{1}$.
Since $\pi$ and $\tau$ preserve order, this monomial also minimizes
$\nu\left(n\right)+m_{1}\gamma_{4}+m_{2}\gamma_{5}+m_{3}\gamma_{0}+m_{4}\gamma_{1}$
and $\nu\left(n\right)+m_{1}\gamma_{2}+m_{2}\gamma_{3}+m_{3}\gamma_{0}+m_{4}\gamma_{1}$.

\begin{gather*}
\textrm{For }	i\neq0, \ \ \ac\left(n\left(a_{i}-b\right)^{m_{1}}\left(a_{0}'-b'\right)^{m_{2}}b^{m_{3}}\left(b'\right)^{m_{4}}\right)\\
	=\ac\left(n\right)\left(\ac\left(\alpha\right)\left(\tilde{a}_{i}-\tilde{a}_{0}\right)\right)^{m_{1}}\left(\tilde{b}'-\tilde{a}_{0}'\right)^{m_{2}}\ac\left(a_{\infty}\right)^{m_{3}}\ac\left(a_{\infty}'\right)^{m_{4}}\\
	=\ac\left(n\right)\left(\tilde{a}_{i}-\tilde{a}_{0}\right)^{m_{1}}\left(\tilde{b}'-\tilde{a}_{0}'\right)^{m_{2}}.\\
	\textrm{Similarly, } \ \ac\left(n\left(a_{0}-b\right)^{m_{1}}\left(a_{0}'-b'\right)^{m_{2}}b^{m_{3}}\left(b'\right)^{m_{4}}\right)\\
	=\ac\left(n\right)\left(\tilde{b}-\tilde{a}_{0}\right)^{m_{1}}\left(\tilde{b}'-\tilde{a}_{0}'\right)^{m_{2}}\ac\left(a_{\infty}\right)^{m_{3}}\ac\left(a_{\infty}'\right)^{m_{4}}\\
	=\ac\left(n\right)\left(\tilde{b}-\tilde{a}_{0}\right)^{m_{1}}\left(\tilde{b}'-\tilde{a}_{0}'\right)^{m_{2}} 
	=\sigma_{i}\left(\ac\left(n\right)\left(\tilde{a}_{i}-\tilde{a}_{0}\right)^{m_{1}}\left(\tilde{b}'-\tilde{a}_{0}'\right)^{m_{2}}\right).\\
\textrm{And for } i,j\neq0 \textrm{ we have } \ \ac\left(n\left(a_{i}-b\right)^{m_{1}}\left(a_{j}'-b'\right)^{m_{2}}b^{m_{3}}\left(b'\right)^{m_{4}}\right)\\
	= \ac\left(n\right)\left(\ac\left(\alpha\right)\left(\tilde{a}_{i}-\tilde{a}_{0}\right)\right)^{m_{1}}\left(\ac\left(\beta\right)\left(\tilde{a}_{j}'-\tilde{a}_{0}'\right)\right)^{m_{2}}\ac\left(a_{\infty}\right)^{m_{3}}\ac\left(a_{\infty}'\right)^{m_{4}}\\
	=\ac\left(n\right)\left(\tilde{a}_{i}-\tilde{a}_{0}\right)^{m_{1}}\left(\tilde{a}_{j}'-\tilde{a}_{0}'\right)^{m_{2}} 
	=\sigma_{j}'\left(\ac\left(n\right)\left(\tilde{a}_{i}-\tilde{a}_{0}\right)^{m_{1}}\left(\tilde{b}'-\tilde{a}_{0}'\right)^{m_{2}}\right).\\
	\textrm{ Since } \  \nu\left(h\left(a_{i}-b,a_{0}'-b',b,b'\right)-n\left(a_{i}-b\right)^{m_{1}}\left(a_{0}'-b'\right)^{m_{2}}b^{m_{3}}\left(b'\right)^{m_{4}}\right)\\
	>\nu\left(h\left(a_{i}-b,a_{0}'-b',b,b'\right)\right), \textrm{ we have}\\
	\ac\left(h\left(a_{i}-b,a_{0}'-b',b,b'\right)\right)=\ac\left(n\left(a_{i}-b\right)^{m_{1}}\left(a_{0}'-b'\right)^{m_{2}}b^{m_{3}}\left(b'\right)^{m_{4}}\right)\\
	=\ac\left(n\right)\left(\tilde{a}_{i}-\tilde{a}_{0}\right)^{m_{1}}\left(\tilde{b}'-\tilde{a}_{0}'\right)^{m_{2}}.\\
	\textrm{Likewise, } \ \nu\left(h\left(a_{0}-b,a_{0}'-b',b,b'\right)-n\left(a_{0}-b\right)^{m_{1}}\left(a_{0}'-b'\right)^{m_{2}}b^{m_{3}}\left(b'\right)^{m_{4}}\right)\\
	>\nu\left(h\left(a_{0}-b,a_{0}'-b',b,b'\right)\right), \textrm{ so}\\
	\ac\left(h\left(a_{0}-b,a_{0}'-b',b,b'\right)\right)=\ac\left(n\left(a_{0}-b\right)^{m_{1}}\left(a_{0}'-b'\right)^{m_{2}}b^{m_{3}}\left(b'\right)^{m_{4}}\right)\\
	=\sigma_{i}\left(\ac\left(n\right)\left(\tilde{a}_{i}-\tilde{a}_{0}\right)^{m_{1}}\left(\tilde{b}'-\tilde{a}_{0}'\right)^{m_{2}}\right).\\
	\textrm{And } \ \nu\left(h\left(a_{i}-b,a_{j}'-b',b,b'\right)-n\left(a_{i}-b\right)^{m_{1}}\left(a_{j}'-b'\right)^{m_{2}}b^{m_{3}}\left(b'\right)^{m_{4}}\right)\\
	>\nu\left(h\left(a_{i}-b,a_{j}'-b',b,b'\right)\right), \textrm{ so }\\
	\ac\left(h\left(a_{i}-b,a_{j}'-b',b,b'\right)\right)=\ac\left(n\left(a_{i}-b\right)^{m_{1}}\left(a_{j}'-b'\right)^{m_{2}}b^{m_{3}}\left(b'\right)^{m_{4}}\right)\\
	=\sigma_{j}'\left(\ac\left(n\right)\left(\tilde{a}_{i}-\tilde{a}_{0}\right)^{m_{1}}\left(\tilde{b}'-\tilde{a}_{0}'\right)^{m_{2}}\right).
\end{gather*}
Since $g\left(a_i,a'_0,b,b'\right)=h\left(a_i-b,a'_0-b',b,b'\right)$, $g\left(a_0,a'_0,b,b'\right)=h\left(a_0-b,a'_0-b',b,b'\right)$, and $g\left(a_i,a'_j,b,b'\right)=h\left(a_i-b,a'_j-b',b,b'\right)$, this is what we wanted to show.

\noindent \textbf{Case 3.} Clear.

\noindent \textbf{Case 4.} $\varphi\left(x;y;z;w;b;b'\right)$ is of the form $\exists u\,\psi\left(x;y;z,u;w;b;b'\right)$,
where $u$ is a variable of sort $k$, and the claim holds for $\psi$.

For $i,j\neq0$, if $\models\varphi\left(a_{i};a_{0}';c;d;b;b'\right)$,
then $\models\psi\left(a_{i};a_{0}';c,e;d;b;b'\right)$ for
some $e\in k$. Then $\models\psi\left(a_{0};a_{0}';\sigma_{i}\left(c\right),\sigma_{i}\left(e\right);\pi\left(d\right);b;b'\right)$
and $\models\psi\left(a_{i};a_{j}';\sigma_{j}'\left(c\right),\sigma_{j}'\left(e\right);\tau\left(d\right);b;b'\right)$,
so $\models\varphi\left(a_{0};a_{0}';\sigma_{i}\left(c\right);\pi\left(d\right);b;b'\right)$
and $\models\varphi\left(a_{i};a_{j}';\sigma_{j}'\left(c\right);\tau\left(d\right);b;b'\right)$.
Note that each of these implications are reversible.

\noindent \textbf{Case 5.} $\varphi\left(x;y;z;w;b;b'\right)$ is of the form $\exists u\,\psi\left(x;y;z;w,u;b;b'\right)$,
where $u$ is a variable of sort $\Gamma_{\infty}$, and the claim
holds for $\psi$. 
 For $i,j\neq0$, if $\models\varphi\left(a_{i};a_{0}';c;d;b;b'\right)$,
then $\models\psi\left(a_{i};a_{0}';c;d,e;b;b'\right)$ for
some $e\in \Gamma_\infty$. Then $\models\psi\left(a_{0};a_{0}';\sigma_{i}\left(c\right);\pi\left(d\right),\pi\left(e\right);b;b'\right)$
and $\models\psi\left(a_{i};a_{j}';\sigma_{j}'\left(c\right);\tau\left(d\right),\tau\left(e\right);b;b'\right)$,
hence $\models\varphi\left(a_{0};a_{0}';\sigma_{i}\left(c\right);\pi\left(d\right);b;b'\right)$  and $\models\varphi\left(a_{i};a_{j}';\sigma_{j}'\left(c\right);\tau\left(d\right);b;b'\right)$.
Since $\pi$ and $\tau$ are bijective, each of these implications
is reversible.	
This concludes the proof of  Claim \ref{cla: val fields 2}, and hence of Theorem \ref{thm: non weak semieq val fields}.

\subsection{Some further applications of Theorem \ref{thm: non weak semieq val fields} and examples}\label{sec: apps of main thm}

\begin{remark}\label{rem: val field more examples}
	Our proof of Theorem \ref{thm: non weak semieq val fields} also applies to any reduct of an $\ac$-valued
field $K$ whose residue field has a non-constant totally indiscernible
sequence to a language $\mathcal{ L}\subseteq\mathcal{ L}_{\text{Denef-Pas}}$
such that $\mathcal{ L}$ contains the relation $\nu\left(x_{1}-y_{1}\right)<\nu\left(x_{2}-y_{2}\right)$,
and every $\mathcal{ L}$-formula is equivalent to a Boolean combination
of $\mathcal{ L}_{\text{Denef-Pas}}$-formulas with no quantifiers of
the main sort. This gives us further examples of NIP theories that
are not weakly semi-equational, such as:
\begin{enumerate}
	\item A Henselian valued field of equicharacteristic  $0$ whose residue field is algebraically
closed.

\item An algebraically closed valued field (of any characteristic).

\item The reduct of either of the above to a valued vector space or
valued abelian group.

\item A generic abstract ultrametric space: a two-sorted structure $\left({\mathcal M},\Gamma_{\infty}\right)$,
with a linear order $\leq$ on $\Gamma_{\infty}$ that is dense with maximal element $\infty\in\Gamma_{\infty}$ and no minimal element, and a function $\nu:{\mathcal M}^{2}\rightarrow\Gamma_{\infty}$,
such that $\nu\left(x,y\right)=\infty\iff x=y$, $\nu\left(x,y\right)=\nu\left(y,x\right)$,
and $\nu\left(x,z\right)\geq\max\left(\nu\left(x,y\right),\nu\left(y,z\right)\right)$,
and such that for every $\gamma\in\Gamma$ and $a\in{\mathcal M}$, there
are $\left(b_{i}\right)_{i\in\mathbb{N}}$ in ${\mathcal M}$ such that
$\nu\left(a,b_{i}\right)=\nu\left(b_{i},b_{j}\right)=\gamma$ for
$i,j\in\mathbb{N}$.
\end{enumerate}
\end{remark}

\begin{expl}
	Let $K$ be a valued field (viewed as a structure in the language of rings with a predicate for the valuation ring $\mathcal{O}$), $d \in \omega$ and let $\mathcal{F}$ be the family of all convex subsets of $K^d$ in the sense of Monna (equivalently, the family of all translates of $\mathcal{O}$-submodules of $K^d$).
	Then $\mathcal{F}$ is a definable family, and a  formula defining it is a semi-equation by \cite[Theorem 4.3]{chernikov2021combinatorial} and Proposition \ref{prop: semieq iff fin breadth}.
\end{expl}

\begin{problem}\label{prob: p-adics}
	Is the field $\mathbb{Q}_{p}$ semi-equational? It is weakly semi-equational by distality.
\end{problem}

\section{Weak semi-equationality in expansions by a predicate}\label{sec: exp by pred}

\subsection{Context}

We recall the setting and some results from \cite{chernikov2013externally} (as usual, below $x,y,z$ denote arbitrary finite tuples of variables).
We start with a theory $T$ in a language $\mathcal{L}$, and let
$\mathcal{L}_{\CP} := \mathcal{L}\cup\left\{ \CP\left(x\right)\right\} $, 
where $\CP$ is a new unary predicate. Let $T_{\CP}:=\Th_{\mathcal{L}_{\CP}}\left(M,A\right)$,
where $A$ is some subset of $M$ (interpreted as $\CP$). We fix some monster model $\left(M',A'\right)\succ\left(M,A\right)$
of $T_{\CP}$. An $\mathcal{L}_{\CP}$-formula  $\psi(x)$ is \emph{bounded} if
it is of the form $Q_{0}y_{0}\in\CP\ldots Q_{n}y_{n}\in\CP\varphi\left(x,y\right)$,
where $Q_{i}\in\left\{ \exists,\forall\right\} $ and $\varphi\left(x,y\right)\in\mathcal{L}$.
We denote the set of all bounded $\mathcal{L}_{\CP}$-formulas by $\mathcal{L}_{\CP}^{\bdd}$
and say that the theory $T_{\CP}$ is \emph{bounded} if every $\mathcal{L}_{\CP}$
formula is equivalent modulo $T_{\CP}$ to a bounded one. Finally,
for $\mathcal{L}\subseteq\mathcal{L}'\subseteq\mathcal{L}_{\CP}\left(M\right)$
we denote by $A_{\indu\left(\mathcal{L}'\right)}$ the $\mathcal{L}'(\emptyset)$-induced
structure on $A$, i.e.~the structure $\left( A; \left( R_{\varphi}(x)   \right)_{\varphi(x) \in \mathcal{L}'} \right)$ with $R_{\varphi} := \left\{ a \in A^{|x|} : \left(M, A \right) \models \varphi(a) \right\}$.
\begin{remark}\label{rem: induced subsets}
\begin{enumerate}
\item 	The structures $A_{\indu\left(\mathcal{L}_{\CP}^{\bdd}\right)}$
and $A_{\indu\left(\mathcal{L}\right)}$ have the same definable subsets
of $A^{n}$, for all $n \in \omega$.  Indeed, given $\psi(x) = Q_{0}y_{0}\in\CP\ldots Q_{n}y_{n}\in\CP\varphi\left(x,y\right)$ with $\varphi(x,y) \in \mathcal{L}_{\CP}^{\bdd}$. Then $R_{\psi}(A)$ can be defined in $A_{\indu\left(\mathcal{L}\right)}$ by $Q_{0}y_{0} \ldots Q_{n}y_{n} R_{\varphi}\left(x,y\right)$.
\item If $T_{\CP}$ is bounded, then clearly $A_{\indu\left(\mathcal{L}_{\CP}\right)}$
and $A_{\indu\left(\mathcal{L}_{\CP}^{\bdd}\right)}$ have the same
definable subsets of $A^{n}$, for all $n$.
\end{enumerate}
\end{remark}

\begin{fact}
\label{fac: NIP is preserved in pairs}
\begin{enumerate}
\item \cite[Corollary 2.5]{chernikov2013externally} Assume that $T$ is
NIP, $A_{\indu\left(\mathcal{L}\right)}$ is NIP and $T_{\CP}$
is bounded. Then $T_{\CP}$ is NIP. 
\item \cite[Corollary 2.6]{chernikov2013externally} In particular, if $T$
is NIP, $A\preceq M$ and $T_{\CP}$ is bounded, then $T_{\CP}$
is NIP.
\end{enumerate}
\end{fact}
\noindent Some results on preservation of equationality under naming a set by a predicate are obtained in \cite{martin2020equational}.
As pointed out in \cite{hieronymi2017distal}, the exact analog with
distality in place of NIP is false:
\begin{fact}
(\cite[Theorem 5.1]{hieronymi2017distal} and the examples after it)
The theory of dense pairs of $o$-minimal structures expanding a group
is not distal (even though it is bounded and the induced structure on the submodel is distal). Their proof shows that the formula $\varphi\left(x,y\right)=\neg\exists u\in\CP\left(x=u+y\right)$
is not a weak semi-equation in the theory of dense pairs.
\end{fact}
In this section we show that at least weak semi-equationality of $T_{\CP}$
can be salvaged. We will need the following properties of indiscernible sequences and definable sets with distal induced structure.
\begin{fact}\cite[Proposition 1.17]{aschenbrenner2020distality}
\label{lem: lifting distality over a predicate}Let $T$ be NIP,
and let $D$ be an $\emptyset$-definable set. Assume that $D_{\indu}$
is distal. Let $\left(c_{i}:i\in\mathbb{Q}\right)$ be an indiscernible
sequence of tuples in $\mathbb{M}$ and let a tuple $b$ from $D$ be given.
Assume that $\left(c_{i}:i\in\mathbb{Q}\setminus\left\{ 0\right\} \right)$
is indiscernible over $b$, then $\left(c_{i}:i\in\mathbb{Q}\right)$
is indiscernible over $b$ as well.
\end{fact}

\begin{lemma}
\label{lem: completing sequence with one point}Assume $T$ is NIP and  
$D$ is an $\emptyset$-definable set with $D_{\indu}$
distal. Let $\left(a_{i}:i\in\mathbb{Q}\right)$ be an $\emptyset$-indiscernible
sequence, $b$ such that $\left(a_{i}:i\in\mathbb{Q}\setminus\left\{ 0\right\} \right)$
is $b$-indiscernible,  and $c\in D$ arbitrary. Then we can find
a sequence $\left(c_{i}:i\in\mathbb{Q}\right)$ such that:
\begin{itemize}
	\item $a_{i}c_{i}\equiv_{b}a_{1}c$
for all $i\in\mathbb{Q}\setminus\left\{ 0\right\} $,

\item  $\left(a_{i}c_{i}:i\in\mathbb{Q}\right)$
is $\emptyset$-indiscernible, and 

\item $\left(a_{i}c_{i}:i\in\mathbb{Q}\setminus\left\{ 0\right\} \right)$
is $b$-indiscernible.
\end{itemize}
\end{lemma}
\begin{proof}
By $b$-indiscernibility of $\left(a_{i}:i\in\mathbb{Q}\setminus\left\{ 0\right\} \right)$, 
Ramsey, compactness and taking automorphisms we can find a sequence $\left(c_{i}:i\in\mathbb{Q}\setminus\left\{ 0\right\} \right)$
in $D$ such that $\left(a_{i}c_{i}:i\in\mathbb{Q}\setminus\left\{ 0\right\} \right)$
is $b$-indiscernible and $a_{i}c_{i}\equiv_{b}a_{1}c$ for all $i\neq0$.
%
It remains to find a $c_{0}\in D$ such that $\left(a_{i}c_{i}:i\in\mathbb{Q}\right)$
is $\emptyset$-indiscernible. Let $I\subseteq\mathbb{Q}\setminus\left\{ 0\right\} $
be an arbitrary finite set and let $\bar{a}_{0}:=\left(a_{i}:i\in I\right)$.
Let $\varepsilon>0$ in $\mathbb{Q}$ be such that $I\subseteq\mathbb{Q}\setminus\left(-\varepsilon,\varepsilon\right)$.
For each $i\in\mathbb{Q}$, let $a_{i}':= (a_{i}, \bar{a}_{0})$ and consider
the sequence $\left(a_{i}':i\in\left(-\varepsilon,\varepsilon\right)\right)$.
It is $\emptyset$-indiscernible since the sequence $\left(a_{i}:i\in\mathbb{Q}\right)$
is, and moreover $\left(a_{i}':i\in\left(-\varepsilon,\varepsilon\right)\setminus\left\{ 0\right\} \right)$
is indiscernible over $\left(c_{i}:i\in I\right)\subseteq D$ (since the sequence of pairs $\left(a_{i}c_{i}:i\in\mathbb{Q}\setminus\left\{ 0\right\} \right)$
is indiscernible). Then
by Fact \ref{lem: lifting distality over a predicate} we have that
$\left(a_{i}':i\in\left(-\varepsilon,\varepsilon\right)\right)$
is indiscernible over $\left(c_{i}:i\in I\right)$. In particular, there exists an automorphism $\sigma$ sending $a'_{\frac{\varepsilon}{2}}$ to $a'_0$ and fixing $\left(c_i : i \in I \right)$; hence sending $a_{\frac{\varepsilon}{2}}$ to $a_0$ and fixing $\left(a_i c_i : i \in I \right)$.
As by assumption
$\left(a_{i}c_{i}:i\in I,i<-\varepsilon\right)+\left(a_{\frac{\varepsilon}{2}}c_{\frac{\varepsilon}{2}}\right)+\left(a_{i}c_{i}:i\in I,i>\varepsilon\right)$
is indiscernible, applying $\sigma$ we have that there is $\tilde{c}_{0} := \sigma\left( c_{\frac{\varepsilon}{2}} \right) \in D$ such that $\left(a_{i}c_{i}:i\in I,i<-\varepsilon\right)+\left(a_{0}\tilde{c}_{0}\right)+\left(a_{i}c_{i}:i\in I,i>\varepsilon\right)$
is indiscernible. As $I$ was arbitrary, we can then find $c_{0}$ as wanted
by compactness.
\end{proof}

\begin{defn}\label{def: almost model compl}
A theory $T_{\CP}$ is \emph{almost model complete} if, modulo $T_{\CP}$,
every $\mathcal{L}_{\CP}$-formula $\psi\left(x\right)$ is equivalent
to a Boolean combination of formulas of the form $\exists y_{0}\in\CP\ldots\exists y_{n-1}\in\CP\varphi\left(x,y\right)$,
where $\varphi\left(x,y\right)$ is an $\mathcal{L}$-formula.
\end{defn}

\begin{theorem}
\label{thm: semi-equational pairs} Assume that $T$ is distal, $A_{\indu\left(\mathcal{L}\right)}$
is distal and $T_{\CP}$ is almost model complete. Then $T_{\CP}$ is weakly 
semi-equational.
\end{theorem}
%

\begin{proof}
We know by Fact \ref{fac: NIP is preserved in pairs} that $T_{\CP}$
is NIP. As $T_{\CP}$ is almost model complete, so in particular
bounded, by Lemma \ref{rem: induced subsets}(1) and (2) 
the structures $A_{\indu\left(\mathcal{L}_{\CP}\right)}$
and $A_{\indu\left(\mathcal{L}\right)}$ have the same definable subsets
of $A^{n}$, for all $n$. Hence the full structure induced on $\CP$
in $T_{\CP}$ is distal, so  Lemma \ref{lem: completing sequence with one point} can be applied to $T_{\CP}$ with $D := \CP$.

Let $\left(M',A'\right)$ be a sufficiently saturated elementary extension
of $\left(M,A\right) \models T_{\CP}$. As $T_{\CP}$ is almost model complete by assumption,
it is enough to show that every formula in $\mathcal{L}_{\CP}$ of
the form 
\[
\varphi\left(x,y\right)=\exists z_{0}\in\CP\ldots\exists z_{n-1}\in\CP\psi\left(x,y,z\right)\text{,}
\]
where $\psi\left(x,x,z\right)\in\mathcal{L}$, is
a weak semi-equation in $T_{\CP}$.

To check Definition \ref{def: weak semi-eq}, assume (using Remark \ref{rem: lin ord dont matter}) that we are
given an $\mathcal{L}_{\CP}$-indiscernible sequence of finite tuples
$\left(a_{i}:i\in\mathbb{Q}\right)$ and a finite tuple $b$, both
in $M'$, such that the sequence $\left(a_{i}:i\in\mathbb{Q}\setminus\left\{ 0\right\} \right)$
is $\mathcal{L}_{P}\left(b\right)$-indiscernible and $\models\varphi\left(a_{i},b\right)$
for all $i\neq0$.
In particular, there is some tuple $c$ in $\CP$ such that $\models\psi\left(a_{1},b,c\right)$
holds.
By Lemma \ref{lem: completing sequence with one point} applied in $T_{\CP}$ with $D := \CP$, it follows
that there is a sequence $\left(c_{i}:i\in\mathbb{Q}\right)$ with
$c_{i}\in\CP$ such that $\left(a_{i}c_{i}:i\in\mathbb{Q}\right)$
is $\mathcal{L}_{\CP}$-indiscernible, $\left(a_{i}c_{i}:i\neq0\right)$
is $\mathcal{L}_{\CP}\left(b\right)$-indiscernible and $a_{i}c_{i}\equiv^{\mathcal{L}_{\CP}}_{b}a_{1}c$
for  $i\neq0$. In particular $\models\psi\left(a_{i},b,c_{i}\right)$
for  $i\neq0$. But $\psi'\left(x,z;y\right):=\psi\left(x,y,z\right)\in\mathcal{L}$
is a semi-equation in $T$ as $T$ is distal, hence  $\models\psi\left(a_{0},b,c_{0}\right)$,
and so $\models\varphi\left(a_{0},b\right)$ holds --- as wanted.
\end{proof}

\begin{cor}
Dense pairs of $o$-minimal structures, as well as the other  examples discussed after \cite[Theorem 5.1]{hieronymi2017distal}, are weakly semi-equational.
\end{cor}


\begin{problem}
(1) In Theorem \ref{thm: semi-equational pairs}, can we relax the assumption to ``$T$ and $A_{\indu(\mathcal{L})}$ are weakly semi-equational''?

(2) Is there an analog of Theorem \ref{thm: semi-equational pairs} for semi-equationality? Even a general result for equationality seems to be missing (the argument in \cite{martin2020equational} for \emph{belles paires} of algebraically closed fields is specific to algebraically closed fields).
	\end{problem}

\bibliographystyle{alpha}
\bibliography{refs}
\end{document}